\documentclass[11pt,a4paper]{article}

\usepackage{amsmath, amsthm, amsfonts, amssymb,color,geometry,enumerate}
\usepackage{graphicx}
\usepackage{fancybox}
\usepackage{cases}
\usepackage{fancyhdr}

\theoremstyle{definition}

\newtheorem{theorem}{Theorem}[section]
\newtheorem{proposition}[theorem]{Proposition}
\newtheorem{lemma}[theorem]{Lemma}
\newtheorem{definition}[theorem]{Definition}
\newtheorem{remark}[theorem]{Remark}
\newtheorem{problem}[theorem]{Problem}
\newtheorem{example}[theorem]{Example}
\newtheorem{conjecture}[theorem]{Conjecture}

\makeatletter

\@addtoreset{equation}{section}
\makeatother

\newcommand{\R}{\mathbb{R}}   
\newcommand{\N}{\mathbb{N}}   
\newcommand{\im}{\text{\normalfont Im}}    
\newcommand{\re}{\text{\normalfont Re}}    
\renewcommand{\epsilon}{\varepsilon}    
\newcommand{\DD}{{\rm D}}           
\newcommand{\supp}{{\rm supp}}   
\newcommand{\ii}{{\rm i}}                

\begin{document}

\title{Large time unimodality for classical and free Brownian motions with initial distributions}
\author{Takahiro Hasebe and Yuki Ueda}
\maketitle

\begin{abstract}
We prove that classical and free Brownian motions with initial distributions are unimodal for sufficiently large time, under some assumption on the initial distributions. The assumption is almost optimal in some sense. Similar results are shown for a symmetric stable process with index 1 and a positive stable process with index $1/2$.  We also prove that free Brownian motion with initial symmetric unimodal distribution is unimodal, and discuss strong unimodality for free convolution.   
\end{abstract}


\tableofcontents


\section{Introduction}

A Borel measure $\mu$ on $\mathbb{R}$ is {\it unimodal} if there exist $a\in\mathbb{R}$ and a function $f\colon\mathbb{R}\rightarrow [0,\infty)$ which is non-decreasing on $(-\infty,a)$ and non-increasing on $(a,\infty)$, such that
\begin{equation}
\mu(dx)=\mu(\{a\})\delta_a+f(x)\,dx. 
\end{equation}  
The most outstanding result on unimodality is Yamazato's theorem \cite{Yamazato} that all classical selfdecomposable distributions are unimodal. After this result, in \cite{HT2}, Hasebe and Thorbj\o rnsen proved the free analog of Yamazato's result: all freely selfdecomposable distributions are unimodal. The unimodality has several other similarity points between classical and free probability theories. However, it is not true that the unimodality shows complete similarity between classical and free probability theories. For example, classical compound Poisson processes are likely to be non-unimodal in large time \cite{Wol78}, while free L\'{e}vy processes with compact support become unimodal in large time \cite{HS}. In this paper, we mainly focus on the unimodality of classical and free Brownian motions with initial distributions and consider whether classical and free versions share similarity points or not. In free probability theory, the semicircle distribution is the free analog of the normal distribution. In random matrix theory, it appears as the limit of eigenvalue distributions of Wigner matrices as the size of random matrices gets close to infinity. Free Brownian motion is defined as a free L\'{e}vy process started at $0$ and distributed with centered semicircle distribution $S(0, t)$ at time $t>0$. Furthermore, we can provide free Brownian motion with initial distribution $\mu$ which is a free L\'{e}vy process distributed with $\mu$ at $t=0$ and with $\mu\boxplus S(0,t)$ at time $t>0$. The definition of $\boxplus$, called free convolution, is in Section \ref{subsec:free convolution}. The additive free convolution is the distribution of the sum of two random variables which are freely independent (the definition of free independence is in \cite{NS}). In \cite{Bia}, Biane gave the density function formula of free Brownian motion with initial distribution (see Section \ref{sec Biane}). In our studies, we first consider the symmetric Bernoulli distribution $\mu:=\frac{1}{2}\delta_{+1}+\frac{1}{2}\delta_{-1}$ as an initial distribution and compute the density function of $\mu\boxplus S(0,t)$. Then we see that the probability distribution $\mu\boxplus S(0,t)$ is unimodal for $t\ge4$ (and it is not unimodal for $0<t<4$). This computation leads to a natural problem:

\begin{problem}
For what class of probability measures on $\mathbb{R}$, free Brownian motion with such an initial distribution becomes unimodal for sufficiently large time?
\end{problem}
We then answer to this problem as follows:

\begin{theorem}\label{thm:main1}
\begin{enumerate}[\rm(1)] 
\item\label{main1-1} Let  $\mu$ be a compactly supported probability measure on $\mathbb{R}$ and $D_\mu:=\sup\{|x-y| : x,y\in \text{ supp}(\mu)\}$. Then $\mu\boxplus S(0,t)$ is unimodal for $t\ge 4D_\mu^2$.

\item Let $f\colon\mathbb{R}\rightarrow [0,\infty)$ be a Borel measurable function. Then there exists a probability measure $\mu$ on $\mathbb{R}$ such that $\mu\boxplus S(0,t)$ is not unimodal for any $t>0$ and
\begin{equation}
\int_\mathbb{R}f(x)\,d\mu(x)<\infty.
\end{equation}
Note that such a measure $\mu$ is not compactly supported from \eqref{main1-1}. 
\end{enumerate}
\end{theorem}
The function $f$ can grow very fast such as $e^{x^2}$, and so, a tail decay of the initial distribution does not imply the large time unimodality.  These results are formulated as Theorem \ref{unimodal1} and Proposition \ref{notunimodal1} in Section \ref{sec FBM}, respectively. 

A natural classical problem arises, that is, for what class of initial distributions on $\mathbb{R}$ Brownian motion becomes unimodal for sufficiently large time $t>0$? 
We prove in this direction the following results (formulated as Theorem \ref{c-unimodal} and Proposition \ref{c-notunimodal} in Section \ref{sec:CBM}, respectively):

\begin{theorem}\label{intro2}
\begin{enumerate}[\rm(1)]
\item Let $\mu$ be a probability measure on $\mathbb{R}$ such that 
\begin{equation}\label{eq:tail}
\alpha:=\int_{\mathbb{R}} e^{\epsilon x^2}\,d\mu(x)<\infty
\end{equation}
for some $\epsilon>0$.  Then the distribution $\mu* N(0,t)$ is unimodal for all $t\ge \frac{36\log(2\alpha)}{\epsilon}$.

\item There exists a probability measure $\mu$ on $\mathbb{R}$ satisfying that
\begin{equation}
\int_{\mathbb{R}} e^{A |x|^p} \,d\mu(x)<\infty \hspace{2mm} \text{ for all $A>0$ and $0< p<2$, and} 
\end{equation}
\begin{equation}
\int_\mathbb{R} e^{Ax^2}d\mu(x)=\infty, \hspace{2mm} \text{ for all } A>0,
\end{equation}
such that $\mu*N(0,t)$ is not unimodal for any $t>0$.
\end{enumerate}
\end{theorem}
Thus, in the classical case, the tail decay \eqref{eq:tail} is almost necessary and sufficient to guarantee the large time unimodality. 
 From these results, we conclude that the classical and free versions of problems on the unimodality for Brownian motions with initial distributions share similarity as both become unimodal in large time, under each assumption.

Further related results in this paper are as follows. In Section \ref{sec Sym} we prove that $\mu \boxplus S(0,t)$ is always unimodal whenever $\mu$ is symmetric around $0$ and unimodal (see Theorem \ref{unimodal2}). In Section \ref{sec SU}, we define freely strongly unimodal probability measures as a natural free analogue of strongly unimodal probability measures.  We conclude that the semicircle distributions are not freely strongly unimodal (Lemma \ref{NSU}). In more general, we have the following result (restated as Theorem \ref{thm:NoSU} later):

\begin{theorem}
Let $\lambda$ be a probability measure with finite variance, not being a delta measure. Then $\lambda$ is not freely strongly unimodal.
\end{theorem}

On the other hand, there are many strongly unimodal distributions with finite variance in classical probability including the normal distributions and exponential distributions. Thus the strong unimodality breaks similarity between classical and free probability theories. 

In Section 5,  we focus on other classical/free L\'{e}vy processes with initial distributions. We consider a symmetric stable process with index $1$ and a positive stable process with index $1/2$. Then we prove large time unimodality similar to the case of Brownian motion but under different tail decay assumptions.


\section{Preliminaries}

\subsection{Definition of free convolution} \label{subsec:free convolution}
Let $\mu$ be a probability measure on $\mathbb{R}$. The Cauchy transform
\begin{equation}
G_\mu(z):=\int_\mathbb{R} \frac{1}{z-x}\,d\mu(x)
\end{equation}
is analytic on the complex upper half plane $\mathbb{C}^+$. We define
the truncated cones
\begin{equation}
\Gamma_{\alpha,\beta}^\pm:=\{z\in\mathbb{C}^\pm: |\re(z)|<\alpha |\im(z)|,\hspace{1mm} \im(z)>\beta\}, \qquad \alpha >0, \beta \in\R. 
\end{equation}
In \cite[Proposition 5.4]{BV}, it was proved that for any $\gamma<0$ there exist $\alpha,\beta>0$ and $\delta<0$ such that $G_\mu$ is univalent on $\Gamma_{\alpha,\beta}^+$ and $\Gamma_{\gamma,\delta}^- \subset G_\mu(\Gamma_{\alpha,\beta}^+)$. Therefore the right inverse function $G_\mu^{-1}$ exists on $\Gamma_{\gamma,\delta}^-$. We define the {\it R-transform of $\mu$} by
\begin{equation}
R_\mu(z):=G_\mu^{-1}(z)-\frac{1}{z}, \hspace{2mm} z\in \Gamma_{\gamma,\delta}^-.
\end{equation}
Then for any probability measures $\mu$ and $\nu$ on $\mathbb{R}$ there exists a unique probability measure $\lambda$ on $\mathbb{R}$ such that 
\begin{equation}
R_\lambda(z)=R_\mu(z)+R_\nu(z),
\end{equation}
for all $z$ in the intersection of the domains of the three transforms. We denote $\lambda:=\mu\boxplus \nu$ and call it the {\it (additive) free convolution} of $\mu$ and $\nu$.

\subsection{Free convolution with semicircle distributions}\label{sec Biane}
The {\it semicircle distribution $S(0,t)$} of mean $0$ and variance $t>0$ is the probability measure with density
\begin{equation}
\frac{1}{2\pi t}\sqrt{4t-x^2}
\end{equation}
on which its support is the interval $[-2\sqrt{t},2\sqrt{t}]$. We then compute its Cauchy transform and its R-transform (see, for example, \cite{NS}):
\begin{equation}
G_{S(0,t)}(z)=\frac{z-\sqrt{z^2-4t}}{2t}, \hspace{2mm} z\in\mathbb{C}^+,
\end{equation}
where the branch of the square root on $\mathbb{C}\setminus \mathbb{R}^+$ is such that $\sqrt{-1}=\ii$. Moreover 
\begin{equation}
R_{S(0,t)}(z)=tz, \hspace{2mm} z\in \mathbb{C}^-.
\end{equation}
Let $\mu$ be a probability measure on $\mathbb{R}$ and let $\mu_t = \mu \boxplus S(0,t)$. Then we define the following set:
\begin{equation}\label{U}
U_t:=\left\{u\in\mathbb{R}~\left|~ \int_\mathbb{R} \frac{1}{(x-u)^2}\,d\mu(x)>\frac{1}{t}\right.\right\},
\end{equation}
and the following function from $\mathbb{R}$ to $[0,\infty)$ by setting
\begin{equation}\label{v}
v_t(u):=\inf\left\{v\ge 0~ \left| ~\int_\mathbb{R} \frac{1}{(x-u)^2+v^2}\,d\mu(x)\le\frac{1}{t}\right.\right\}.
\end{equation}
Then $v_t$ is continuous on $\mathbb{R}$ and one has $U_t=\{x\in\mathbb{R}\mid v_t(u)>0\}$. Moreover, for every $u \in U_t$, $v_t(u)$ is the unique solution $v >0$ of the equation
\begin{equation}\label{eqv}
\int_{\mathbb{R}}\frac{1}{(x-u)^2+v^2}\,d\mu(x)=\frac{1}{t}.
\end{equation}
We define 
\begin{equation}
\Omega_{t,\mu}:=\{x+\ii y\in\mathbb{C}\mid y>v_t(x)\}.
\end{equation}
Then the map $H_t(z):=z+R_{S(0,t)}(G_\mu(z))$ is a homeomorphism from $\overline{\Omega_{t,\mu}}$ to $\mathbb{C}^+\cup\mathbb{R}$ and it is conformal from $\Omega_{t,\mu}$ onto $\mathbb{C}^+$. Hence there exists its inverse function $F_t:\mathbb{C}^+\cup\mathbb{R}\rightarrow \overline{\Omega_{t,\mu}}$ such that $F_t(\mathbb{C}^+)= \Omega_{t,\mu}$. Moreover, the domain $\Omega_{t,\mu}$ is equal to the connected component of the set $H_t^{-1}(\mathbb{C}^+)$ which contains $\ii y$ for large $y>0$. By \cite[Lemma 1, Proposition 1]{Bia}, for all $z\in\mathbb{C}^+$ one has 
\begin{equation}
\begin{split}
G_{\mu_t}(z) &= G_{\mu_t}(H_t(F_t(z)))\\
			&= G_{\mu_t}\Bigl(F_t(z)+R_{S(0,t)}\bigl(G_\mu(F_t(z))\bigr)\Bigr)\\
			&= G_{\mu}(F_t(z)).
\end{split}
\end{equation}
Moreover $G_{\mu_t}$ has a continuous extension to $\mathbb{C}^+\cup \mathbb{R}$. Then $\mu_t$ is Lebesgue absolutely continuous by Stieltjes inversion and the density function $p_t\colon\mathbb{R}\rightarrow [0,\infty)$ of $\mu_t$ is given by 
\begin{equation}\label{eq density}
p_t(\psi_t(x))=-\frac{1}{\pi}\lim_{y\rightarrow+0} \im(G_{\mu_t}(\psi_t(x)+\ii y))=\frac{v_t(x)}{\pi t}, \hspace{2mm}  x\in\mathbb{R},
\end{equation}
where 
\begin{equation}
\psi_t(x)=H_t(x+\ii v_t(x))=x+t\int_\mathbb{R} \frac{(x-u)}{(x-u)^2+v_t(x)^2}\,d\mu(u). 
\end{equation}
It can be shown that $\psi_t$ is a homeomorphism of $\mathbb{R}$.
Furthermore the topological support of $p_t$ is given by $\psi_t(\overline{U_t})$,  and $p_t$ extends to a continuous function on $\R$ and is real analytic in $\{x \in \R: p_t(x)>0\}$.

\begin{example}\label{Ber}
Let $\mu:=\frac{1}{2}\delta_{-1}+\frac{1}{2}\delta_{+1}$ be the symmetric Bernoulli distribution. Recall that the topological support of the density function $p_t$ is equal to $\psi_t(\overline{U_t})$. If $0< t\leq1$, we have 
\begin{equation}\label{Ut2}
U_t=\left\{u\in \mathbb{R} ~\left|~
   \sqrt{\frac{2+t-\sqrt{t^2+8t}}{2}}<|u|<\sqrt{\frac{2+t+\sqrt{t^2+8t}}{2}}
 \right.\right\}, 
 \end{equation}
and if $t> 1$, we have 
\begin{equation}\label{Ut1}
U_t=\left\{u\in\mathbb{R}~\left|~ -\sqrt{\frac{2+t+\sqrt{t^2+8t}}{2}}<u<\sqrt{\frac{2+t+\sqrt{t^2+8t}}{2}}\right.\right\}.
 \end{equation}
 If $0<t\leq 1$ then the number of connected components of $U_t$ is two, and therefore $\mu_t$ is not unimodal. If $t>1$ then $U_t$ is connected. By the density formula \eqref{eq density}, the unimodality of $\mu_t$ is equivalent to the unimodality of the function $v_t(u)$, the latter being expressed in the form
\begin{equation}\label{vt}
v_t(u)=\sqrt{\frac{-(2u^2+2-t)+\sqrt{t^2+16u^2}}{2}}. 
\end{equation}
Calculating its first derivative then shows that $\mu_t$ is unimodal if and only if $t\ge 4$; see Figures \ref{FBM0.1}--\ref{FBM6}.

\end{example}

According to Example \ref{Ber}, we have a question about what class of initial distributions implies the large time unimodality of free Brownian motion. We will give a result for this in Section \ref{sec FBM}.

\begin{figure}[h]
\begin{center}
\begin{minipage}{0.3\hsize}
\begin{center}
\includegraphics[width=40mm,clip]{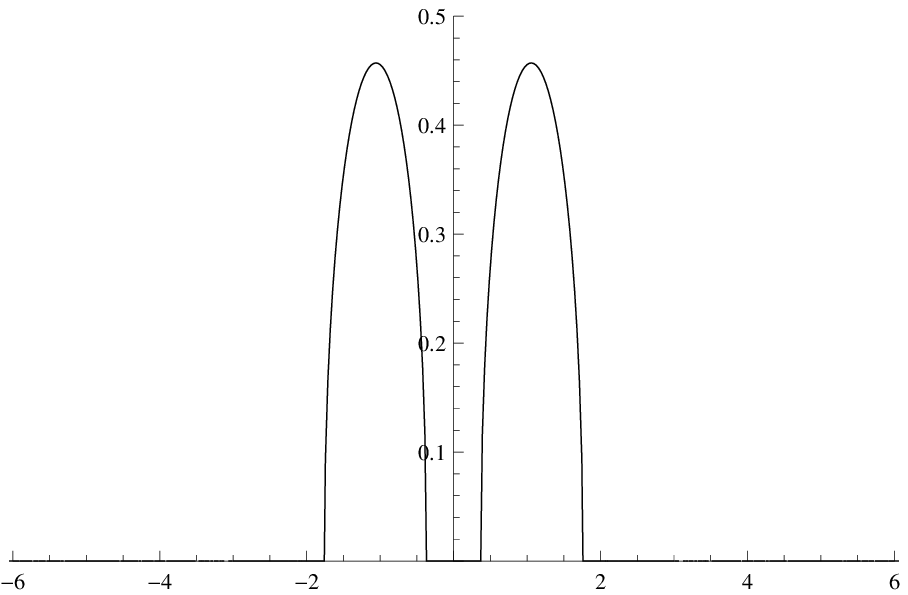}
\caption{$p_{0.25}(x)$} \label{FBM0.1}
\end{center}
  \end{minipage}
\begin{minipage}{0.3\hsize}
\begin{center}
\includegraphics[width=40mm,clip]{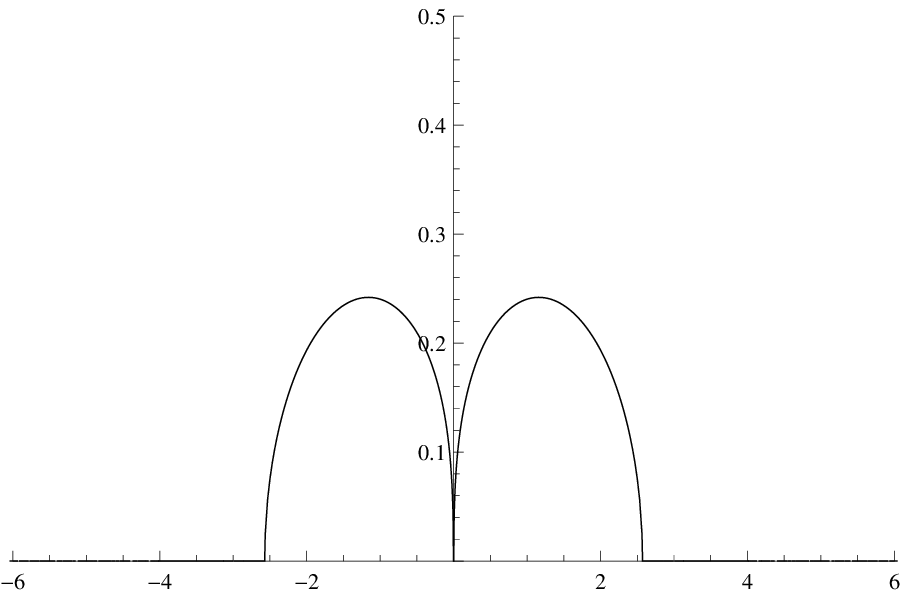}
\caption{$p_1(x)$}
\end{center}
\end{minipage}
\begin{minipage}{0.3\hsize}
\begin{center}
\includegraphics[width=40mm,clip]{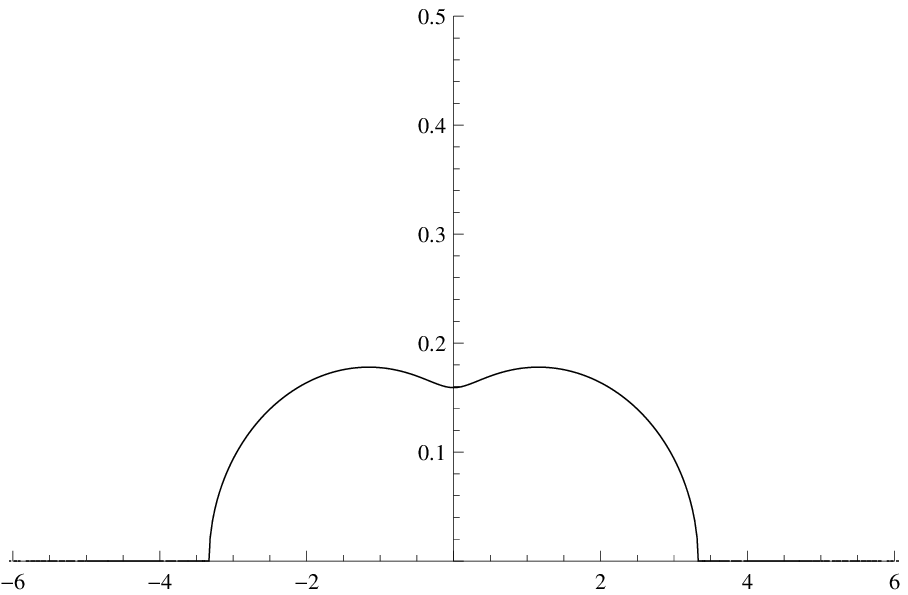}
\caption{$p_2(x)$}
\end{center}
\end{minipage}
\end{center}

\begin{center}
\begin{minipage}{0.3\hsize}
\begin{center}
\includegraphics[width=40mm,clip]{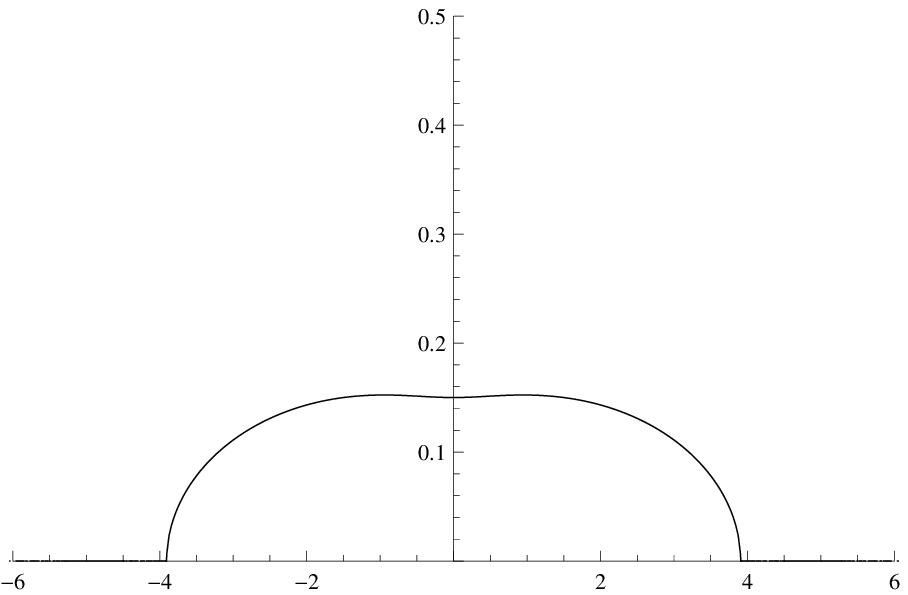}
\caption{$p_3(x)$}
\end{center}
\end{minipage}
\begin{minipage}{0.3\hsize}
\begin{center}
\includegraphics[width=40mm,clip]{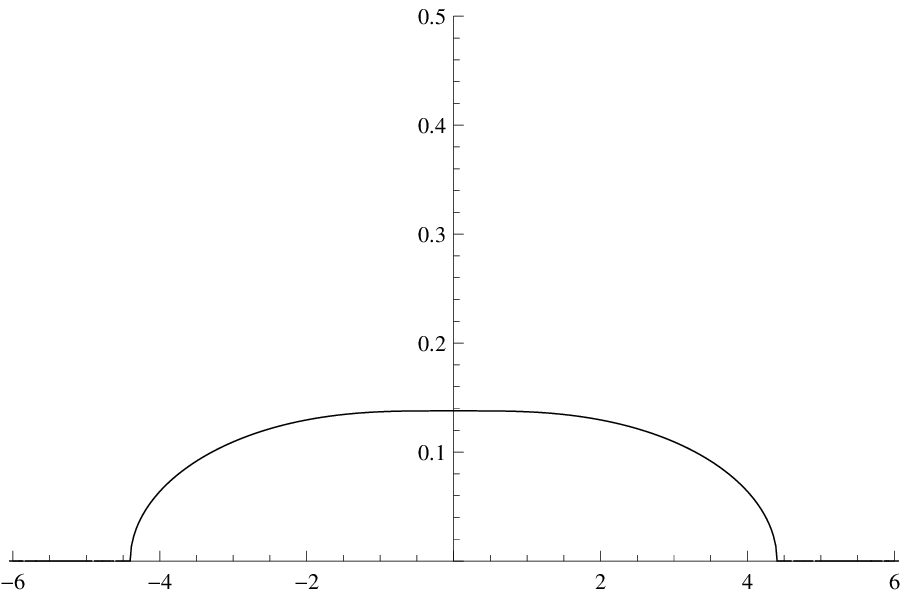}
\caption{$p_4(x)$}\label{FBM5}
\end{center}
\end{minipage}
\begin{minipage}{0.3\hsize}
\begin{center}
\includegraphics[width=40mm,clip]{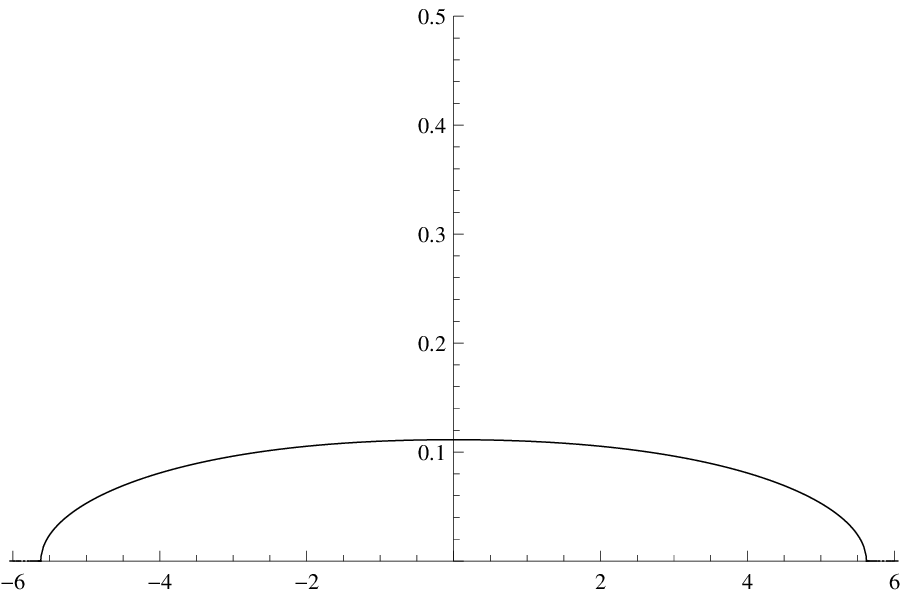}
\caption{$p_7(x)$}\label{FBM6}
\end{center}
\end{minipage}
\end{center}
\end{figure}

\subsection{A basic lemma for unimodality} 

A unimodal distribution is allowed to have a plateau or a discontinuous point in its density such as the uniform distribution or the exponential distribution. If we exclude such cases, then unimodality can be characterized in terms of levels of density. 

\begin{lemma}\label{lem basic unimodal}
Let $\mu$ be a probability measure on $\R$ that is Lebesgue absolutely continuous. Suppose that its density $p(x)$  extends to a continuous function on $\R$ and is real analytic in $\{x\in\R: p(x)>0\}$. Then $\mu$ is unimodal if and only if, for any $a>0$, the equation $p(x)=a$ has at most two solutions $x$.  
\end{lemma}
The proof is just to use the intermediate value theorem and the fact that a real analytic function never has plateaux. Note that the idea of the above lemma was first introduced by Haagerup and Thorbj{\o}rnsen \cite{HT14} to prove that a ``free gamma distribution'' is unimodal.


\section{Free Brownian motion with initial distribution}

\subsection{Large time unimodality for free Brownian motion}\label{sec FBM}

The semicircular distribution is the free analogue of the normal distribution. Therefore, a free L\'{e}vy process is called {\it (standard) free Brownian motion with initial distribution $\mu$} if its distribution at time $t$ is given by $\mu\boxplus S(0,t)$.  In this section, we firstly prove that free Brownian motion with compactly supported initial distribution is unimodal for sufficiently large time, by using Biane's density formula (see Section \ref{sec Biane}). 

\begin{lemma}\label{lem unimodal}
Suppose that $t>0$. If for any $R>0$ the equation 
\begin{equation}\label{sol1}
\int_{\mathbb{R}}\frac{1}{(x-u)^2+R^2}\,d\mu(x)=\frac{1}{t}
\end{equation} has at most two solutions $u\in \mathbb{R}$, then $\mu \boxplus S(0,t)$ is unimodal.
\end{lemma}
\begin{proof} We adopt the notations $\mu_t= \mu \boxplus S(0,t), v_t, p_t$ and $\psi_t$ in Section \ref{sec Biane}. This proof is similar to \cite[Proposition 3.8]{HT2}. 
Recall that $\mu_t$ is absolutely continuous with respect to Lebesgue measure and the density function $p_t$ is continuous on $\mathbb{R}$ since its Cauchy transform $G_{\mu_t}$ has a continuous extension to $\mathbb{C}^+\cup \mathbb{R}$, and the density function is real analytic in $\psi_t(U_t)$ and continuous on $\R$ by Section \ref{sec Biane}. Since $\psi_t$ is a homeomorphism of $\R$, by Lemma \ref{lem basic unimodal} it suffices to show that for any $a>0$ the equation 
\begin{equation}
a=p_t(\psi_t(u))=\frac{v_t(u)}{\pi t}, \hspace{2mm} u\in \mathbb{R}
\end{equation}
has at most two solutions $u \in \R$.  
Then  
\begin{equation}
\left\{u\in\mathbb{R}~\left| ~a=\frac{v_t(u)}{\pi t} \right.  \right\}=\left\{u\in\mathbb{R}~\left| ~ \int_\mathbb{R}\frac{1}{(x-u)^2+(\pi a t)^2}\,d\mu(x) =\frac{1}{t}\right.\right\}
\end{equation}
since $v_t(u)$ is a unique solution of the equation \eqref{eqv}, if it is positive. By the assumption, the equation $a=\frac{v_t(u)}{\pi t}$ has at most two solutions in $\mathbb{R}$. 
\end{proof}

Now we are ready to prove Theorem \ref{thm:main1} \eqref{main1-1}. 

\begin{theorem}\label{unimodal1}
Let $\mu$ be a compactly supported probability measure, and let $D_\mu$ be the diameter of the support:
$
D_\mu = \sup\{|x-y|: x, y\in \supp(\mu)\}. 
$
 Then $\mu \boxplus S(0,t)$ is unimodal for $t\geq 4D_\mu^2$.
\end{theorem}

\begin{proof}
Applying a shift we may assume that $\mu$ is supported on $[-\frac{D_\mu}{2},\frac{D_\mu}{2}].$ 
For fixed $R>0$, we define by
\begin{equation}\label{phiR}
\xi_R(u):=\int_{\mathbb{R}}\frac{1}{(x-u)^2+R^2}\,d\mu(x), \hspace{2mm} u\in \mathbb{R}.
\end{equation}Then we have
\begin{equation}\label{dphiR}
\begin{split}
\xi_R'(u)=\int_{-\frac{D_\mu}{2}}^{\frac{D_\mu}{2}}\frac{2(x-u)}{\{(x-u)^2+R^2\}^2}\,d\mu(x),\\
\xi_R''(u)=\int_{-\frac{D_\mu}{2}}^{\frac{D_\mu}{2}}\frac{6(x-u)^2-2R^2}{\{(x-u)^2+R^2\}^3}\,d\mu(x).
\end{split}
\end{equation}
If $u\le -\frac{D_\mu}{2}$, then $x-u\ge-\frac{D_\mu}{2}-(-\frac{D_\mu}{2})=0$, so that $\xi_R'(u)>0$. If $u\ge \frac{D_\mu}{2}$, then $x-u\le \frac{D_\mu}{2}-\frac{D_\mu}{2}=0$, so that $\xi_R'(u)<0$. We take $t \ge 4D_\mu^2$ and consider the form of $\xi_R(u)$ on $u\in (-\frac{D_\mu}{2},\frac{D_\mu}{2})$.\\
If $0<R<\sqrt{3}D_\mu$, then $(x-u)^2+R^2<(\frac{D_\mu}{2}+\frac{D_\mu}{2})^2+(\sqrt{3}D_\mu)^2=4D_\mu^2 \le t$, so that 
\begin{equation}\label{ine1}
\xi_R(u) = \int_{\mathbb{R}}\frac{1}{(x-u)^2+R^2} \,d\mu(x)>\frac{1}{t},\hspace{2mm} u\in \Bigl(-\frac{D_\mu}{2},\frac{D_\mu}{2}\Bigr).
\end{equation}
Hence the equation $\xi_R(u)=\frac{1}{t}$ has at most two solutions $u\in \R$ if $t\ge 4D_\mu^2$. If $R\ge \sqrt{3}D_\mu$, then the function $\xi_R(u)$ satisfies the following three conditions:
\begin{itemize}
\item $\xi_R'(u)>0$ for all $u<-\frac{D_\mu}{2}$ and $\xi_R'(u)<0$ for all $u>\frac{D_\mu}{2}$,
\item $\xi_R'$ is continuous on $\mathbb{R}$,
\item $\xi_R''(u)<0$ for all $u\in \Bigl(-\frac{D_\mu}{2},\frac{D_\mu}{2}\Bigr)$.
\end{itemize}
By intermediate value theorem, $\xi_R'$ has a unique zero in $(-\frac{D_\mu}{2},\frac{D_\mu}{2})$ and hence $\xi_R$ takes a unique local maximum in the interval $(-\frac{D_\mu}{2},\frac{D_\mu}{2})$.  Therefore the equation $\xi_R(u)=\frac{1}{t}$ has at most two solutions if $t\ge 4D_\mu^2$.  Hence we have that the free convolution $\mu\boxplus S(0,t)$ is unimodal if $t\ge 4D_\mu^2$ by Lemma \ref{lem unimodal}.
\end{proof}

\begin{problem}
What is the optimal universal constant $C>0$ such that $\mu \boxplus S(0,t)$ is unimodal for all $t\ge C D_\mu^2$ and all probability measures $\mu$ with compact support? 
\end{problem}

We have already shown that $C\leq 4$. Recall from Example \ref{Ber} that if $\mu$ is the symmetric Bernoulli distribution on $\{-1,1\}$, then $\mu \boxplus S(0,t)$ is unimodal if and only if $t\ge 4$. Since the diameter $D_\mu$ of the support of $\mu$ is 2, we conclude that $2 \leq C  \leq 4$.

The next question is whether there exists a probability measure $\mu$ such that $\mu \boxplus S(0,t)$ is not unimodal for any $t>0$ or at least for sufficiently large $t>0$. Such a distribution must have an unbounded support if it exists. 
Such an example can be constructed with an idea similar to \cite[Proposition 4.13]{Hua} (see also \cite[Example 5.3]{HS}).

\begin{proposition}\label{notunimodal1}
Let $f\colon \R \to [0,\infty)$ be a Borel measurable function. Then there exists a probability measure $\mu$ such that $\mu \boxplus S(0,t)$ is not unimodal at any $t>0$ and 
\begin{equation}\label{f-moment}
\int_\R f(x) \,d\mu(x) <\infty. 
\end{equation}
\end{proposition}
\begin{proof} We follow the notations in Section \ref{sec Biane}. 
Let $\{w_n\}_n$ and $\{a_n\}_n$ be sequences in $\mathbb{R}$ satisfying
\begin{itemize}
\item $w_n>0, \sum_{n=1}^\infty w_n=1, $
\item $a_{n+1}>a_n, \hspace{2mm} n\ge1,$ 
\item $\displaystyle\lim_{n\rightarrow\infty} (a_{n+1}-a_n)=\infty.$
\end{itemize}
Consider the probability measure $\mu:=\sum_{n=1}^\infty w_n\delta_{a_n}$ on $\mathbb{R}$ and define the function
\begin{equation}\label{phimu}
X_{\mu}(u):=\int_{\mathbb{R}}\frac{1}{(u-x)^2}\,d\mu(x)=\sum_{n=1}^\infty \frac{w_n}{(u-a_n)^2}.
\end{equation}
Recall that $U_t$ is the set of $u\in\R$ such that $X_{\mu}(u) >1/t$. Set $b_k:=\frac{a_{k+1} + a_{k}}{2}$ for all $k\in \mathbb{N}$. Then we have
\begin{equation}\label{xkan}
|b_k-a_n|\ge \frac{a_{k+1}-a_k}{2}, \hspace{3mm} k,n\in \mathbb{N}
\end{equation}
and so 
$$
X_{\mu}(b_k)\le \left(\frac{2}{a_{k+1}-a_k}\right)^2 \sum_{n=1}^\infty w_n=\left(\frac{2}{a_{k+1}-a_k}\right)^2\rightarrow 0
$$ as $k\rightarrow\infty$. This means that for each $t>0$ there exists an integer $K(t)>0$ such that $X_{\mu}(b_k)<\frac{1}{t}$ for all $k\ge K(t)$. This implies that, for $k \geq K(t)$,  the closure of $U_t$ does not contain $b_k$. 
Therefore, there exists a sequence $\{\epsilon_k(t)\}_{k =K(t)+1}^\infty$ of positive numbers such that 
\begin{equation}\label{suppmut}
U_t \cap (b_{K(t)}, \infty) = \bigcup_{k=K(t)+1}^\infty (a_k-\epsilon_k(t), a_k+\epsilon_k(t)),  
\end{equation}
where the closures of the intervals are disjoint. Thus the set $\supp(\mu_t) = \psi_t(\overline{U_t})$ consists of infinitely many connected components for all $t>0$, and hence $\mu_t$ is not unimodal for any $t>0$. 

In the above construction, the weights $w_n$ are only required to be positive. Therefore, we may take 
\begin{equation}
w_n = \frac{c}{n^2\max\{f(a_n), 1\}},
\end{equation}
 where $c>0$ is a normalizing constant. Then we can show the integrability condition \eqref{f-moment}. 
\end{proof}
Proposition \ref{notunimodal1} shows that tail decay estimates of the initial distribution do not guarantee the large time unimodality. 
 In Section \ref{sec:CBM} we see that the situation is different for classical Brownian motion.

\subsection{Free Brownian motion with symmetric initial distribution}\label{sec Sym}
This section proves a unimodality result of different flavor. 
It is well known that if $\mu$ and $\nu$ are symmetric unimodal distributions, then $\mu\ast \nu$ is also (symmetric) unimodal (see \cite[Exercise 29.22]{Sat}). The free analogue of this statement is not known. 
\begin{conjecture}
Let $\mu$ and $\nu$ be symmetric unimodal distributions. Then $\mu\boxplus \nu$ is unimodal. 
\end{conjecture}

We can give a positive answer in the special case when one distribution is a semicircle distribution.

\begin{theorem}\label{unimodal2} Let $\mu$ be a symmetric unimodal distribution on $\R$. Then $\mu\boxplus S(0,t)$ is unimodal for any $t>0$. 
\end{theorem}
\begin{remark}
There is a unimodal probability measure $\mu$ such that $\mu \boxplus S(0,1)$ is not unimodal (see Lemma \ref{NSU}). Hence we cannot remove the assumption of symmetry of $\mu$. 
\end{remark}
\begin{proof} By Lemma \ref{lem unimodal} it suffices to show that for any $R>0$ the equation
\begin{equation}\label{eq unimodal}
\int_\R \frac{1}{(u-x)^2+R^2}\,d\mu(x)  = \frac{1}{t} 
\end{equation}
has at most two solutions $u \in \R$. Up to a constant multiplication, the LHS is the density of $\mu* C_R$ which is unimodal, since the Cauchy distribution $C_R$ is symmetric unimodal, and is real analytic. Therefore, Lemma \ref{lem basic unimodal} implies that the equation \eqref{eq unimodal} has at most two solutions. This shows that $\mu \boxplus S(0,t)$ is unimodal. 
\end{proof}

\subsection{Freely strong unimodality} \label{sec SU}
In classical probability, a probability measure is said to be {\it strongly unimodal} if $\mu\ast \nu$ is unimodal for all unimodal distributions $\nu$ on $\R$. A distribution is strongly unimodal if and only if the distribution is Lebesgue absolutely continuous, supported on an interval and a version of its density is log-concave (see \cite[Theorem 52.3]{Sat}). From this characterization, the normal distributions are strongly unimodal. We discuss the free version of strong unimodality.

\begin{definition}
A probability measure $\mu$ on $\R$ is said to be {\it freely strongly unimodal} if $\mu\boxplus \nu$ is unimodal for all unimodal distributions $\nu$ on $\R$. 
\end{definition}

\begin{lemma}\label{NSU}
The semicircle distributions are not freely strongly unimodal.  
\end{lemma}
\begin{proof}We follow the notations $v_t, p_t$ and $\psi_t$ in Section \ref{sec Biane}. 
The Cauchy distribution is not strongly unimodal since its density is not log-concave. Hence, there exists a probability measure $\mu$ such that $\mu\ast C_1$ is not unimodal. Since the density of $\mu\ast C_1 $ is real analytic on $\R$, from Lemma \ref{lem basic unimodal} there exists $t >0$ such that the equation
\begin{equation}
\pi \frac{d (\mu\ast C_1)}{d x} = \int_\R \frac{1}{(x-y)^2+1}\,d\mu(y)  = \frac{1}{t}
\end{equation}
has at least three distinct solutions $x_1,x_2,x_3 \in\R$. This shows that $v_{t}(x_i) = 1$ for $i=1,2,3$ and so the density $p_{t}$ of $\mu\boxplus S(0,t)$ satisfies that $p_{t} (\psi_{t}(x_i)) = \frac{1}{\pi t}$ for $i=1,2,3$. Hence $S(0,t) \boxplus \mu$ is not unimodal. By changing the scaling, we conclude that all semicircle distributions are not freely strongly unimodal. 
\end{proof}

\begin{theorem}\label{thm:NoSU} Let $\lambda$ be a probability measure with finite variance, not being a delta measure. Then $\lambda$ is not freely strongly unimodal. 
\end{theorem}
\begin{proof} We may assume that $\lambda$ has mean 0. Suppose that $\lambda$ is freely strongly unimodal. A simple induction argument shows that $\lambda^{\boxplus n}$ is freely strongly unimodal for all $n\in\N$.  We derive a contradiction below. By Lemma \ref{NSU} we can take a unimodal  probability measure $\mu$ such that $\mu\boxplus S(0,1)$ is not unimodal. Our hypothesis shows that the measure $\DD_{\sqrt{n v}}(\mu)\boxplus\lambda^{\boxplus n}$ is unimodal, where $v$ is the variance of $\lambda$  and $\DD_c(\rho)$ is the push-forward of a measure $\rho$ by the map $x\mapsto c x$ for $c\in \R$. By the free central limit theorem, the unimodal distributions $\DD_{1/\sqrt{nv}}(\DD_{\sqrt{n v}}(\mu)\boxplus\lambda^{\boxplus n}) = \mu\boxplus \DD_{1/\sqrt{nv}}(\lambda^{\boxplus n})$ weakly converge to $\mu\boxplus S(0,1)$ as $n \to \infty$. Since the set of unimodal distributions is weakly closed, we conclude that $\mu\boxplus S(0,1)$ is unimodal, a contradiction. 
\end{proof}

\begin{problem}
Does there exist a freely strongly unimodal probability measure  that is not a delta measure? 
\end{problem}



\section{Classical Brownian motion with initial distribution}\label{sec:CBM}
This section discusses large time unimodality for {\it standard Brownian motion with an initial distribution $\mu$}, which is a L\'evy process with the distribution $\mu* N(0,t)$ at time $t \geq0$. 
Before going to a general case, one example will be helpful in understanding unimodality for $\mu \ast N(0,t)$. 

\begin{example}
Elementary calculus shows that 
\begin{equation}
\frac{1}{2}(\delta_{-1} + \delta_1) \ast N(0,t) 
\end{equation}
is unimodal if and only if $t\geq 1$; see Figures \ref{CBM0.25}-\ref{CBM4}. 
\end{example}

\begin{figure}[h]
\begin{center}
\begin{minipage}{0.3\hsize}
\begin{center}
\includegraphics[width=40mm,clip]{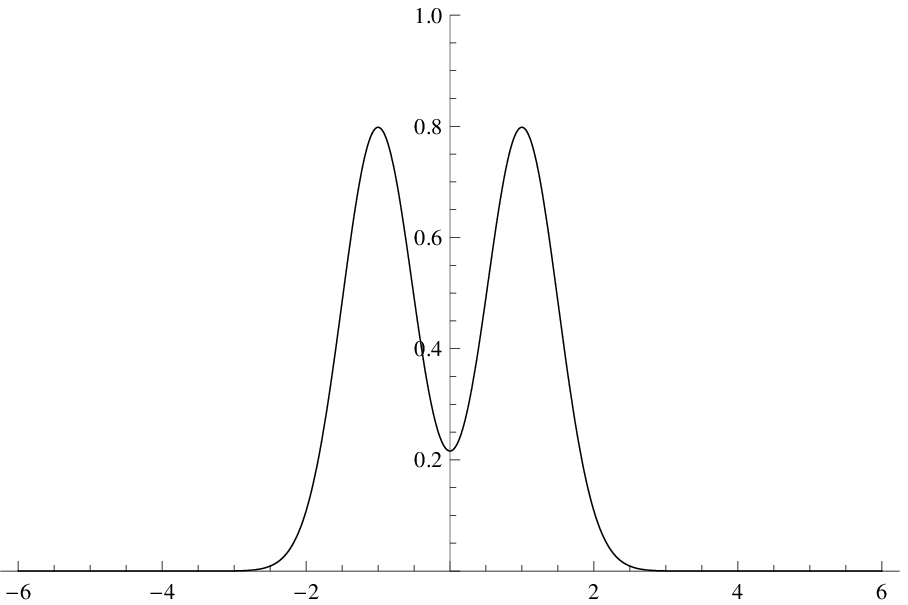}
\caption{$t=0.25$} \label{CBM0.25}
\end{center}
  \end{minipage}
\begin{minipage}{0.3\hsize}
\begin{center}
\includegraphics[width=40mm,clip]{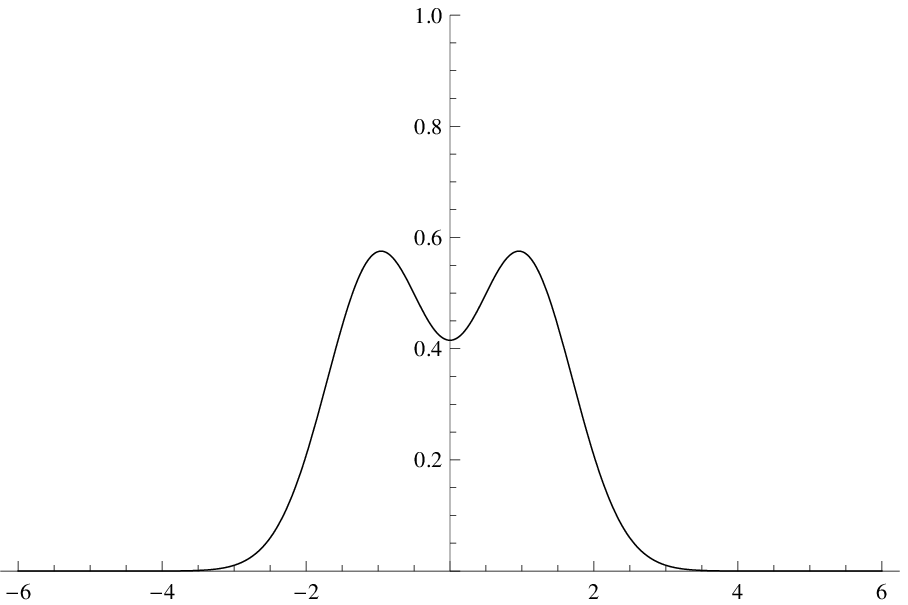}
\caption{$t=0.5$} 
\end{center}
\end{minipage}
\begin{minipage}{0.3\hsize}
\begin{center}
\includegraphics[width=40mm,clip]{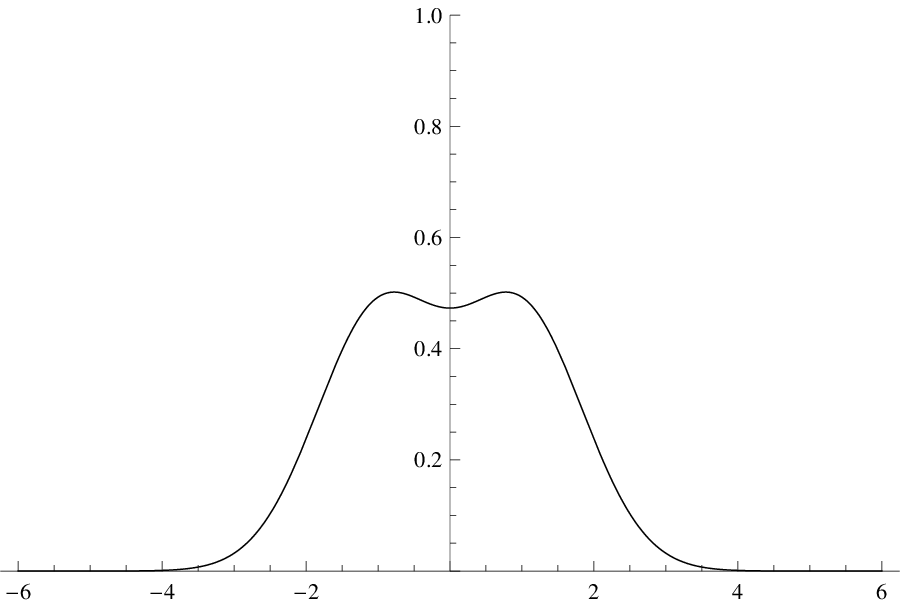}
\caption{$t=0.75$} 
\end{center}
\end{minipage}
\begin{minipage}{0.3\hsize}
\begin{center}
\includegraphics[width=40mm,clip]{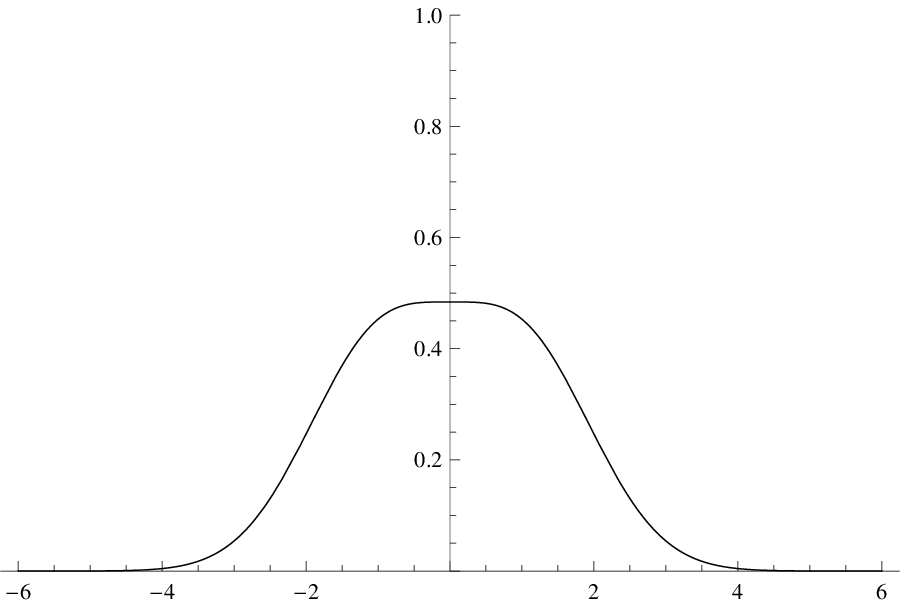}
\caption{$t=1$}
\end{center}
\end{minipage}
\begin{minipage}{0.3\hsize}
\begin{center}
\includegraphics[width=40mm,clip]{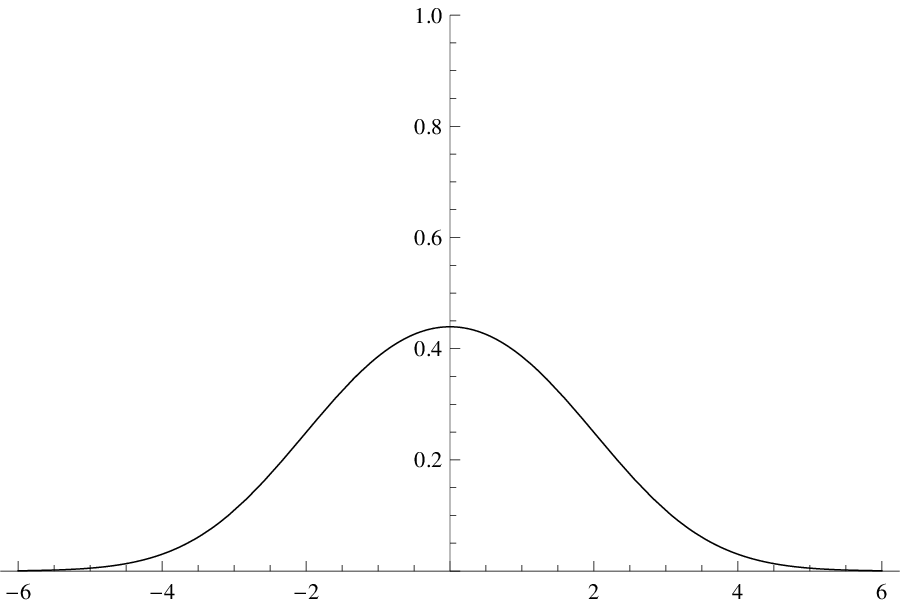}
\caption{$t=2$}
\end{center}
\end{minipage}
\begin{minipage}{0.3\hsize}
\begin{center}
\includegraphics[width=40mm,clip]{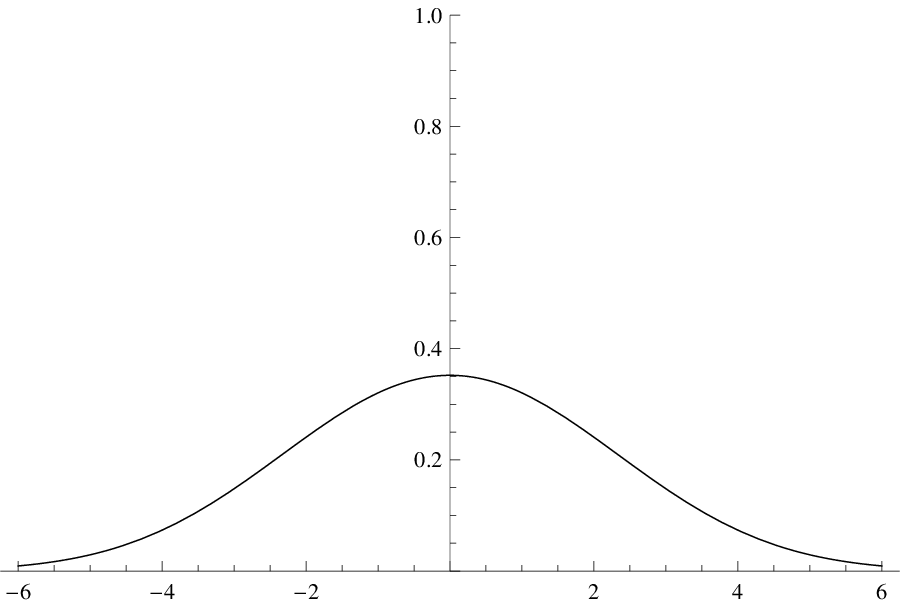}
\caption{$t=4$}\label{CBM4}
\end{center}
\end{minipage}
\end{center}
\end{figure}

This simple example suggests that Brownian motion becomes unimodal for sufficiently large time, possibly under some condition on the initial distribution.  In some sense, we give an almost optimal condition on the initial distribution for the large time unimodality to hold. 
We start from providing an example of initial distribution with which the Brownian motion never becomes unimodal. 
\begin{proposition}\label{c-notunimodal} 
There exists a probability measure $\mu$ on $\mathbb{R}$ such that $\mu * N(0,t)$ is not unimodal at any $t>0$, and 
\begin{equation}\label{c-0}
\int_\mathbb{R} e^{A |x|^p} d \mu(x)<\infty
\end{equation}
for all $A >0$ and $0<p <2$. 
\end{proposition}

\begin{proof}
We consider sequences $\{w_n\}_{n\in\mathbb{N}} \subset (0,\infty)$ and  $\{a_n\}_{n\in\mathbb{N}} \subset \mathbb{R}$ such that 
\begin{align}
&\sum_{n=1}^\infty w_n=1, \label{c-01} \\
&\displaystyle b_k:=\inf_{n\in\N\setminus\{k\}} |a_k-a_n-1|\rightarrow \infty \text{~as~} k\rightarrow \infty. \label{c-02}
\end{align}
Moreover we assume that for all $t>0$ there exists $k_0=k_0(t)\in\mathbb{N}$ such that 
\begin{equation}\label{c-1}
k\ge k_0(t) \Rightarrow w_ke^{-\frac{1}{2t}}>b_ke^{-\frac{b_k^2}{2t}}.
\end{equation}
Define a probability measure $\mu$ by setting $\mu:=\sum_{n=1}^\infty w_n \delta_{a_n}$ and the following function:
\begin{equation}
f_t(x):=\sqrt{2\pi t} \cdot \frac{d(\mu*N(0,t))}{dx}(x)=\sum_{n=1}^\infty w_n e^{-\frac{(x-a_n)^2}{2t}}, \hspace{2mm} x\in \mathbb{R}.
\end{equation}
Then we have
\begin{equation}
f_t'(x)=\sum_{n=1}^\infty w_n \Bigl( -\frac{x-a_n}{t}\Bigr) e^{-\frac{(x-a_n)^2}{2t}}, \hspace{2mm} x\in \mathbb{R}.
\end{equation}
If needed we may replace $k_0(t)$ by a larger integer so that $b_k > \sqrt{t}$ holds for all $k \geq k_0(t)$. For $k \geq k_0(t)$ we have 
\begin{equation}
\begin{split}
f_t'(a_k-1) &= \frac{w_k}{t} e^{-\frac{1}{2t}} - \sum_{n\in\mathbb{N},n\neq k} w_n \Bigl( \frac{a_k-a_n-1}{t}\Bigr)e^{-\frac{(a_k-a_n-1)^2}{2t}}\\
              &\ge \frac{w_k}{t} e^{-\frac{1}{2t}} - \sum_{n\in\mathbb{N},n\neq k} w_n \frac{|a_k-a_n-1|}{t} e^{-\frac{(a_k-a_n-1)^2}{2t}}\\
              &\ge \frac{w_k}{t} e^{-\frac{1}{2t}} - \sum_{n\in\mathbb{N},n\neq k} w_n \frac{b_k}{t} e^{-\frac{b_k^2}{2t}}\\
              &\ge \frac{1}{t}\Bigl(w_k e^{-\frac{1}{2t}} - b_k e^{-\frac{b_k^2}{2t}}\Bigr) > 0, 
\end{split}
\end{equation}
where we used the fact that $x\mapsto |x| e^{-\frac{x^2}{2t}}$ takes the global maximum at $x=\pm\sqrt{t}$ on the third inequality and the assumption \eqref{c-1} on the last line. 
This implies that $\mu*N(0,t)$ is not unimodal for any $t>0$. 

Next we take specific sequences $\{w_n\}_n$ and $\{a_n\}_n$ satisfying the conditions \eqref{c-01}--\eqref{c-1}. Set $a_k=a^k$, $k\in\mathbb{N}$ where $a\geq2$. Note that there exists some constant $c>0$ such that $b_k\ge ca^k$ for all $k\in\mathbb{N}$. Hence we have that $b_k \rightarrow \infty$ as $k\rightarrow \infty$. Moreover we set
\begin{equation}
w_k:=Me^{-\frac{\delta b_k^2}{k}},
\end{equation}
where $\delta>0$ and $M>0$ is a normalized constant, that is, $M=(\sum_{k=1}^\infty e^{-\frac{\delta b_k^2}{k}})^{-1}$ (note that the series converges). Since $b_k\rightarrow \infty$ as $k\rightarrow \infty$, we have
\begin{equation}
\frac{b_ke^{-\frac{b_k^2}{2t}}}{w_k}=\frac{1}{M}b_ke^{\frac{\delta b_k^2}{k}-\frac{b_k^2}{2t}}\rightarrow 0, \hspace{2mm} \text{as } k\rightarrow \infty.
\end{equation}
For all $t>0$ there exists $k_0=k_0(t)\in\mathbb{N}$ such that $k\ge k_0$ implies that $e^{-\frac{1}{2t}}>b_ke^{-\frac{b_k^2}{2t}}/w_k$. Therefore the sequences $\{w_n\}_n$ and $\{a_n\}_n$ satisfy the condition \eqref{c-1}.
Finally we show that $\mu=\sum_{n=1}^\infty w_n \delta_{a_n}$ has the property \eqref{c-0}. For every $A>0$ and $0<p<2$, using the inequality $b_k \geq c a^k$ shows that
\begin{equation}
\begin{split}
\int_\mathbb{R} e^{A |x|^p} \,d\mu(x) &= \sum_{k=1}^\infty w_k e^{A |a_k|^p}= M\sum_{k=1}^\infty e^{-\delta k^{-1}b_k^2}e^{A a^{kp}}\\
                 &\le M\sum_{k=1}^\infty e^{-A a^{pk}(\delta c^2 A^{-1}k^{-1} a^{(2-p)k}-1)} < \infty.
\end{split}
\end{equation}
Thus the proof is complete.
\end{proof}

Note that for any sequences $\{w_n\}_n$ and $\{a_n\}_n$ satisfying the conditions \eqref{c-01}--\eqref{c-1} the distribution $\mu=\sum_{n=1}^\infty w_n \delta_{a_n}$ has the following property:
\begin{equation}\label{c-2}
\int_\mathbb{R} e^{\epsilon x^2} \,d\mu (x) =\infty \text{~~for all~~} \epsilon>0. 
\end{equation}
Then we have a natural question whether the large time unimodality of $\mu \ast N(0,t)$ holds or not if the initial distribution $\mu$ does not satisfy the condition \eqref{c-2}. We solve this question as  follows.

\begin{theorem}\label{c-unimodal} 
Let $\mu$ be a probability measure on $\mathbb{R}$ such that
\begin{equation}\label{exp-moment}
\alpha:=\int_{\mathbb{R}} e^{\epsilon x^2} \,d\mu(x)<\infty .
\end{equation}
for some $\epsilon>0$. Then $\mu \ast N(0,t)$ is unimodal for $t\ge\frac{36\log(2\alpha)}{\epsilon}$.
\end{theorem}
\begin{remark}
The proof becomes much easier if we assume that $\mu$ has a compact support. 
\end{remark}
\begin{proof}
For $x>0$ we have
\begin{equation}\label{Markov ineq}
\mu(|y|>x)=\int_{|y|>x}1 \,d\mu(y) \le \int_{|y|>x} \frac{e^{\epsilon y^2}}{e^{\epsilon x^2}} \,d\mu(y) \le \alpha e^{-\epsilon x^2}.
\end{equation}
Let
\begin{equation}
f_t(x):=\sqrt{2\pi t}\cdot\frac{d(\mu\ast N(0,t))}{dx}(x)=\int_\mathbb{R}e^{-\frac{(x-y)^2}{2t}}\,d\mu(y).
\end{equation}
Then we have
\begin{equation}
\begin{split}
f_t'(x)&=\int_\mathbb{R}\Bigl(\frac{y-x}{t}\Bigr)e^{-\frac{(y-x)^2}{2t}}\,d\mu(y)\\
&=\int_{x}^{\infty} \frac{y-x}{t}e^{-\frac{(y-x)^2}{2t}}\,d\mu(y)-\int_{-\infty}^{x}\frac{x-y}{t}e^{-\frac{(x-y)^2}{2t}}\,d\mu(y).
\end{split}
\end{equation}
For $x>0$, we have
\begin{equation}
\int_{x}^{\infty} \frac{y-x}{t}e^{-\frac{(y-x)^2}{2t}}\,d\mu(y)\le \frac{\sqrt{t}}{t}e^{-\frac{1}{2}}\mu((x,\infty))\le \frac{\alpha}{\sqrt{t}}e^{-\frac{1}{2}-\epsilon x^2}.
\end{equation}
For $x\ge \frac{\sqrt{t}}{2}$, we have
\begin{equation}
\begin{split}
\int_{-\infty}^{x}\frac{x-y}{t}e^{-\frac{(x-y)^2}{2t}}\,d\mu(y) &\ge \int_{-3x}^{x-\frac{\sqrt{t}}{3}} \frac{x-y}{t} e^{-\frac{(x-y)^2}{2t}}\,d\mu(y)\\
&\ge \frac{1}{t}\min\Bigl\{\frac{\sqrt{t}}{3}e^{-\frac{1}{18}}, 4xe^{-\frac{8x^2}{t}}\Bigr\} \mu\Biggl( \Bigl[-3x,x-\frac{\sqrt{t}}{3}\Bigr]\Biggr).
\end{split}
\end{equation}
Now the function $g(x):=4xe^{-\frac{8x^2}{t}}$ has a local maximum at $x=\pm\frac{\sqrt{t}}{4}$. For all $x\ge \frac{\sqrt{t}}{2}$,
\begin{equation}
g(x) \le 2e^{-2}\sqrt{t} < \frac{1}{3}e^{-\frac{1}{18}}\sqrt{t},
\end{equation}
where $2e^{-2} \approx 0.2706$ and $\frac{1}{3}e^{-\frac{1}{18}} \approx 0.3153$. Hence we have
\begin{equation}
\begin{split}
\int_{-\infty}^{x}\frac{x-y}{t}e^{-\frac{(x-y)^2}{2t}}\,d\mu(y)
&\ge \frac{4x}{t}e^{-\frac{8x^2}{t}} \mu\Biggl( \Bigl[-3x,x-\frac{\sqrt{t}}{3}\Bigr]\Biggr)\\
&\ge \frac{4x}{t}e^{-\frac{8x^2}{t}}  \mu\Biggl( \Bigl[-\frac{\sqrt{t}}{6},\frac{\sqrt{t}}{6}\Bigr]\Biggr)\\
&\ge \frac{2}{\sqrt{t}}e^{-\frac{8x^2}{t}}(1-\alpha e^{-\frac{\epsilon t}{36}}), 
\end{split}
\end{equation}
where the last inequality holds thanks to \eqref{Markov ineq} and $x\ge \frac{\sqrt{t}}{2}$. Therefore if $x\ge \frac{\sqrt{t}}{2}$ then
\begin{equation}
\begin{split}
f_t'(x)&\le \frac{\alpha}{\sqrt{t}}e^{-\frac{1}{2}-\epsilon x^2} -\frac{2}{\sqrt{t}}e^{-\frac{8x^2}{t}}(1-\alpha e^{-\frac{\epsilon t}{36}}) \\
&= \frac{\alpha}{\sqrt{t}}e^{-\frac{1}{2}-\epsilon x^2} \Biggl\{ 1-\frac{2e^{\frac{1}{2}}}{\alpha}e^{\epsilon x^2- \frac{8x^2}{t}}(1-\alpha e^{-\frac{\epsilon t}{36}}) \Biggr\}.
\end{split}
\end{equation}
If $t\ge \frac{16}{\epsilon}$ then $e^{\epsilon x^2- \frac{8x^2}{t}} \ge e^{\frac{1}{2}\epsilon x^2} \ge e^{\frac{\epsilon t}{8}}$. Hence if $t \ge \frac{16}{\epsilon}$ then
\begin{equation}
f_t'(x) \le  \frac{\alpha}{\sqrt{t}}e^{-\frac{1}{2}-\epsilon x^2} \Biggl\{ 1-\frac{2e^{\frac{1}{2}}}{\alpha}e^{\frac{\epsilon t}{8}}  (1-\alpha e^{-\frac{\epsilon t}{36}}) \Biggr\}.
\end{equation}
If $t\ge \frac{36\log(2\alpha)}{\epsilon}$ then $1-\alpha e^{-\frac{\epsilon t}{36}} \ge \frac{1}{2}$ and  we have
\begin{equation}
1-\frac{2e^{\frac{1}{2}}}{\alpha}e^{\frac{\epsilon t}{8}} (1-\alpha e^{-\frac{\epsilon t}{36}})\le 1-e^{\frac{1}{2}}\cdot 2^{\frac{9}{2}} \cdot \alpha^{\frac{7}{2}}<0.
\end{equation}
By taking $t\ge \max \{\frac{16}{\epsilon}, \frac{36\log(2\alpha)}{\epsilon}\}=\frac{36\log(2\alpha)}{\epsilon}$, we have $f_t'(x)<0$ if $x\ge \frac{\sqrt{t}}{2}$. Then we have
\begin{equation}
f_t'(x)<0 \hspace{2mm}\text{ if } x\ge \frac{\sqrt{t}}{2}.
\end{equation}
Similarly, we have
\begin{equation}
f_t'(x)>0 \hspace{2mm}\text{ if } x\le -\frac{\sqrt{t}}{2}.
\end{equation}

Next, we will show that $f_t''(x)<0$ for all $x \in \mathbb{R}$ with $|x|<\frac{\sqrt{t}}{2}$.
We then calculate the following
\begin{equation}
\begin{split}
f_t''(x)&=\int_\mathbb{R} \frac{(x-y)^2-t}{t^2} e^{-\frac{(x-y)^2}{2t}}\,d\mu(y)\\
&=\int_{|x-y|>\sqrt{t}}\frac{(x-y)^2-t}{t^2} e^{-\frac{(x-y)^2}{2t}}\,d\mu(y)-\int_{|x-y|\le \sqrt{t}}\frac{t-(x-y)^2}{t^2} e^{-\frac{(x-y)^2}{2t}}\,d\mu(y).
\end{split}
\end{equation}
Since the function $h(u):=\frac{u^2-t}{t^2}e^{-\frac{u^2}{2t}}$ has a local maximum at $u=\pm \sqrt{3t}$, we have
\begin{equation}
\int_{|x-y|>\sqrt{t}}\frac{(x-y)^2-t}{t^2} e^{-\frac{(x-y)^2}{2t}}\,d\mu(y)\le \frac{2}{t}e^{-\frac{3}{2}}\mu([x-\sqrt{t},x+\sqrt{t}]^c).
\end{equation}
For all $x\in \mathbb{R}$ with $|x|<\frac{\sqrt{t}}{2}$, we have $[x-\sqrt{t},x+\sqrt{t}]^c \subset \bigl[-\frac{\sqrt{t}}{2},\frac{\sqrt{t}}{2}\bigr]^c$, and therefore
\begin{equation}
\int_{|x-y|>\sqrt{t}}\frac{(x-y)^2-t}{t^2} e^{-\frac{(x-y)^2}{2t}}\,d\mu(y)\le \frac{2}{t}e^{-\frac{3}{2}}\mu \Bigl(\Bigl[-\frac{\sqrt{t}}{2},\frac{\sqrt{t}}{2}\Bigr]^c \Bigr) \le  \frac{2}{t}e^{-\frac{3}{2}} \alpha e^{-\frac{\epsilon t}{4}}.
\end{equation}
Since the function $-h(u)$ is decreasing on $[0,\sqrt{3t}]$, we have
\begin{equation}
\begin{split}
\int_{|x-y|\le \sqrt{t}}\frac{t-(x-y)^2}{t^2}e^{-\frac{(x-y)^2}{2t}}\,d\mu(y) &\ge \int_{|x-y|\le \frac{2\sqrt{t}}{3}}\frac{t-(x-y)^2}{t^2}e^{-\frac{(x-y)^2}{2t}}\,d\mu(y)\\
&\ge \frac{t-\frac{4}{9}t}{t^2} e^{-\frac{\frac{4}{9}t}{2t}} \mu\Bigl(\Bigl[x-\frac{2\sqrt{t}}{3}, x+\frac{2\sqrt{t}}{3} \Bigr] \Bigr)\\
&= \frac{5}{9t}e^{-\frac{2}{9}}\mu\Bigl(\Bigl[x-\frac{2\sqrt{t}}{3}, x+\frac{2\sqrt{t}}{3} \Bigr] \Bigr)\\
&\ge  \frac{5}{9t}e^{-\frac{2}{9}} \mu\Bigl( \Bigl[-\frac{\sqrt{t}}{6},\frac{\sqrt{t}}{6}\Bigr]\Bigr)\\
&\ge  \frac{5}{9t}e^{-\frac{2}{9}} (1-\alpha e^{-\frac{\epsilon t}{36}}).
\end{split}
\end{equation}
Therefore
\begin{equation}
\begin{split}
f_t''(x)&\le  \frac{2}{t}e^{-\frac{3}{2}} \alpha e^{-\frac{\epsilon t}{4}} -\frac{5}{9t}e^{-\frac{2}{9}} (1-\alpha e^{-\frac{\epsilon t}{36}})\\
&=\frac{1}{t} \Bigl\{   2e^{-\frac{3}{2}} \alpha e^{-\frac{\epsilon t}{4}} -\frac{5}{9}e^{-\frac{2}{9}} (1-\alpha e^{-\frac{\epsilon t}{36}})  \Bigr\}.
\end{split}
\end{equation}
If $t\ge \frac{36\log(2\alpha)}{\epsilon}$, then 
\begin{equation}
2e^{-\frac{3}{2}} \alpha e^{-\frac{\epsilon t}{4}} -\frac{5}{9}e^{-\frac{2}{9}} (1-\alpha e^{-\frac{\epsilon t}{36}}) \le \frac{e^{-\frac{3}{2}}}{(2\alpha)^8}
-\frac{5}{18}e^{-\frac{2}{9}}<0,
\end{equation}
and therefore $f_t''(x)<0$ if $|x|< \frac{\sqrt{t}}{2}$. Thus, if $t\ge \frac{36\log(2\alpha)}{\epsilon}$, then we have the following properties:
\begin{equation}
\begin{split}
x\le -\frac{\sqrt{t}}{2} &\Rightarrow f_t'(x)>0,\\
x\ge \frac{\sqrt{t}}{2} &\Rightarrow f_t'(x)<0,\\
|x|<\frac{\sqrt{t}}{2} &\Rightarrow f_t''(x)<0.
\end{split}
\end{equation}
Hence $f_t(x)$ has a unique local maximum point in $\Bigl( -\frac{\sqrt{t}}{2},\frac{\sqrt{t}}{2} \Bigr)$ when $t\ge\frac{36\log(2\alpha)}{\epsilon}$. Therefore $\mu* N(0,t)$ is unimodal for $t\ge\frac{36\log(2\alpha)}{\epsilon}$.
\end{proof}

\begin{remark} 
If $\mu$ is unimodal then $\mu \ast N(0,t)$ is unimodal for all $t>0$.  This is a consequence of the strong unimodality of the normal distribution $N(0,t)$ (see Section \ref{sec SU}),  in contrast with the failure of freely strong unimodality of the semicircle distribution (see Lemma \ref{NSU}). 
\end{remark}

 We close this section by placing a problem for future research. 
 
\begin{problem} Estimate the position of the mode of classical Brownian motion with initial distributions satisfying the assumption \eqref{exp-moment}. Our proof shows that for $t \geq \frac{36}{\epsilon} \log (2\alpha)$, the mode is located in the interval $[-\sqrt{t}/2, \sqrt{t}/2]$.  How about free Brownian motion?  
\end{problem}


\section{Large time unimodality for stable processes with index $1$ and index $1/2$ with initial distributions} 
We investigate large time unimodality for stable processes with index $1$ (following Cauchy distributions) and index $1/2$ (following L\'evy distributions) with initial distributions.

\subsection{Cauchy process with initial distribution}\label{sec:Cauchy}
Let $\{C_t\}_{t\ge0}$ be the symmetric Cauchy distribution 
\begin{equation}
C_t(dx):=\frac{t}{\pi(x^2+t^2)}\cdot1_{\mathbb{R}}(x)\,dx,\hspace{2mm} x\in\mathbb{R}, \qquad C_0=\delta_0, 
\end{equation}
which forms both classical and free convolution semigroups. A {\it Cauchy process with initial distribution $\mu$} is a L\'evy process that follows the law $\mu \ast C_t$ at time $t\geq0$. 
It is known that the Cauchy distribution satisfies the identity
\begin{equation}\label{f-Cauchy}
\mu\boxplus C_t= \mu*C_t  
\end{equation}
for any $\mu$ and $t\geq0$, and so the distributions $\mu \ast C_t$ can also be realized as the laws at time $t \geq0$ of a free L\'evy process with initial distribution $\mu$. 
As in the cases of free and classical BMs, taking $\mu$ to be the symmetric Bernoulli $\frac{1}{2}(\delta_{-1} + \delta_1)$ is helpful. By calculus we can show that 
$\mu \ast C_t$ is unimodal if and only if $t\geq \sqrt{3}$; see Figures \ref{CP0.25}-\ref{CP3}. 
Thus it is again natural to expect that a Cauchy process becomes unimodal for sufficiently large time, under some condition on the initial distribution.  
We start from the following counterexample. 

\begin{figure}[h]
\begin{center}
\begin{minipage}{0.3\hsize}
\begin{center}
\includegraphics[width=40mm,clip]{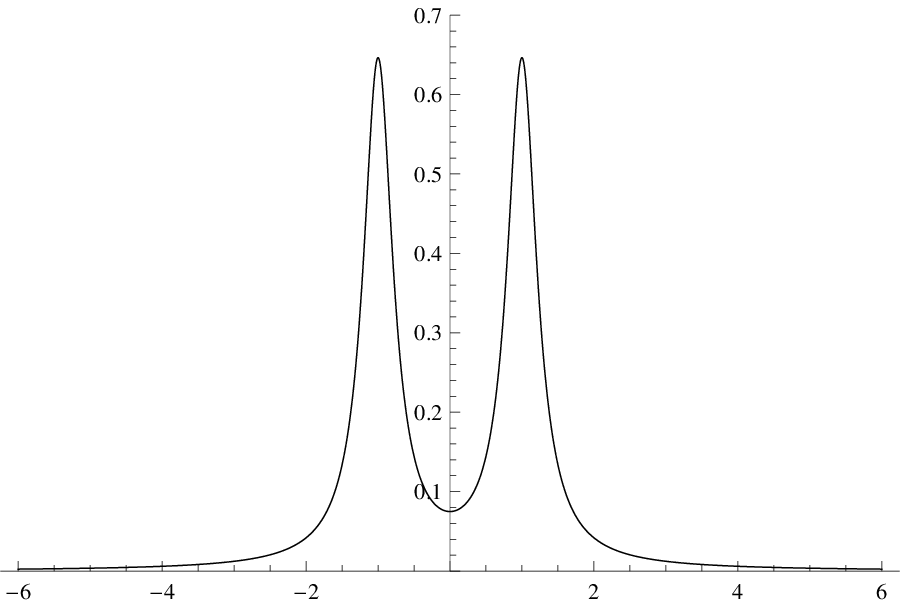}
\caption{$t=0.25$} \label{CP0.25}
\end{center}
  \end{minipage}
\begin{minipage}{0.3\hsize}
\begin{center}
\includegraphics[width=40mm,clip]{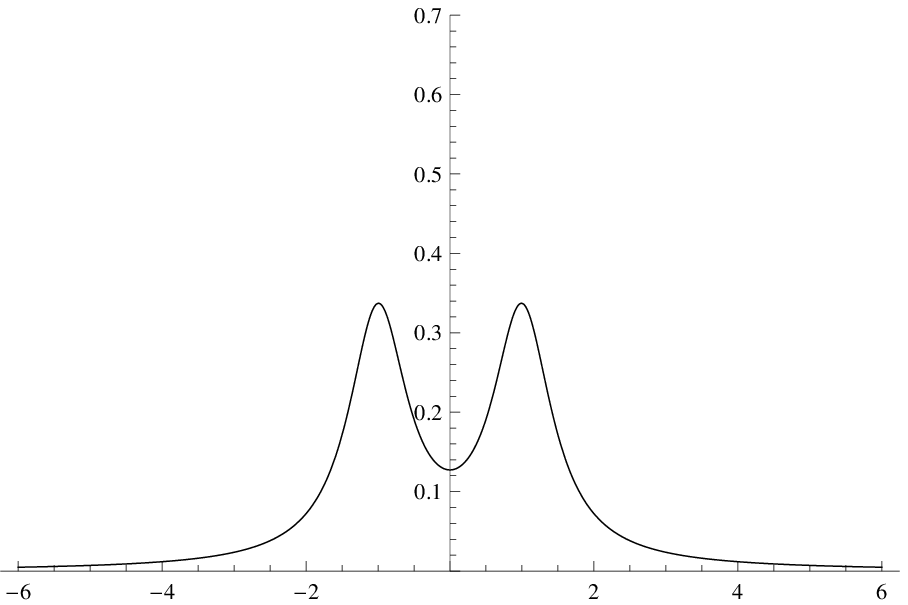}
\caption{$t=0.5$}
\end{center}
\end{minipage}
\begin{minipage}{0.3\hsize}
\begin{center}
\includegraphics[width=40mm,clip]{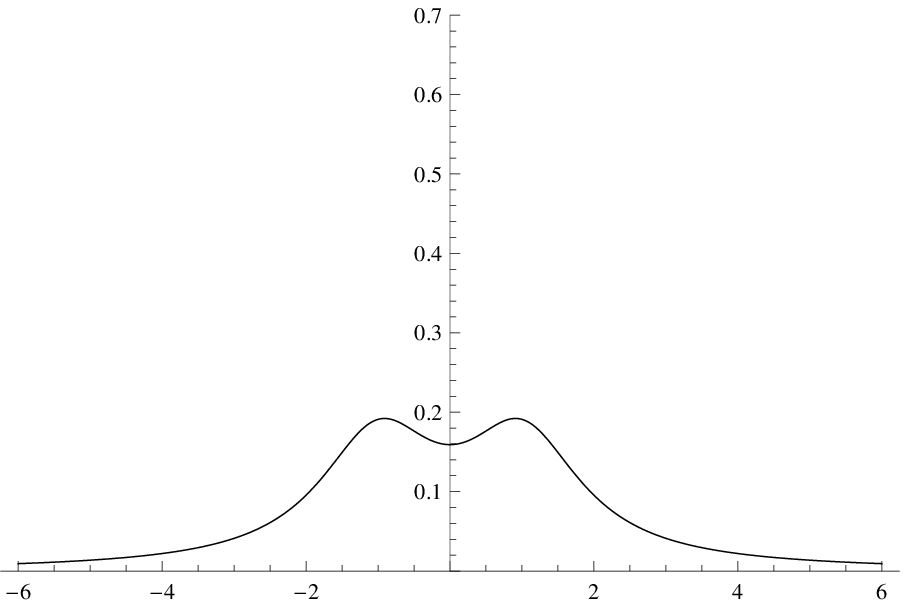}
\caption{$t=1$}
\end{center}
\end{minipage}
\begin{minipage}{0.3\hsize}
\begin{center}
\includegraphics[width=40mm,clip]{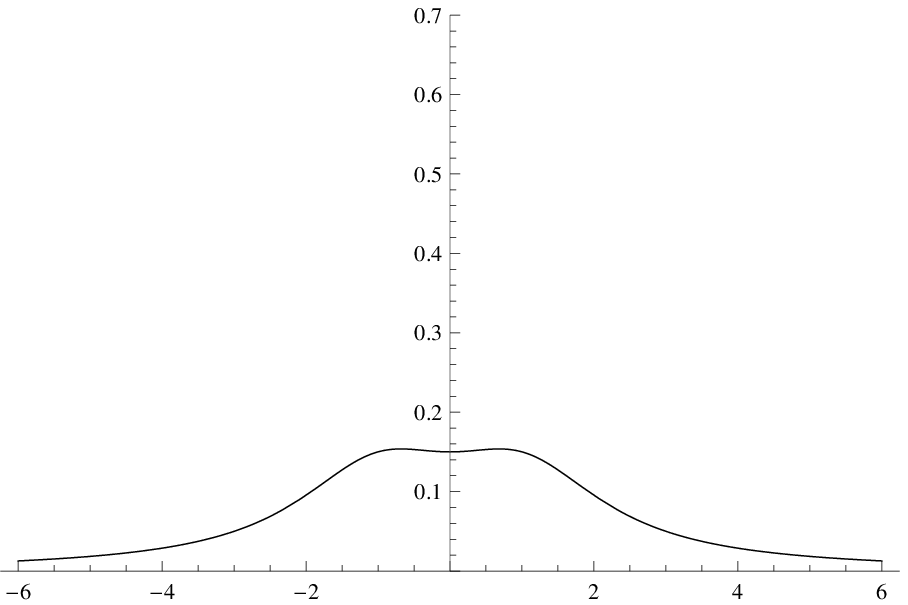}
\caption{$t=\sqrt{2}$}
\end{center}
  \end{minipage}
\begin{minipage}{0.3\hsize}
\begin{center}
\includegraphics[width=40mm,clip]{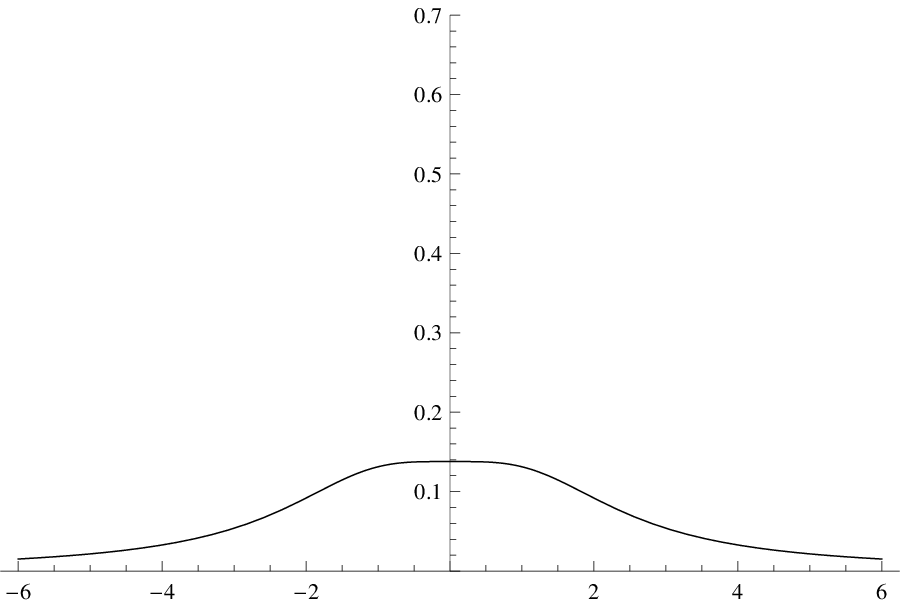}
\caption{$t=\sqrt{3}$}
\end{center}
  \end{minipage}
\begin{minipage}{0.3\hsize}
\begin{center}
\includegraphics[width=40mm,clip]{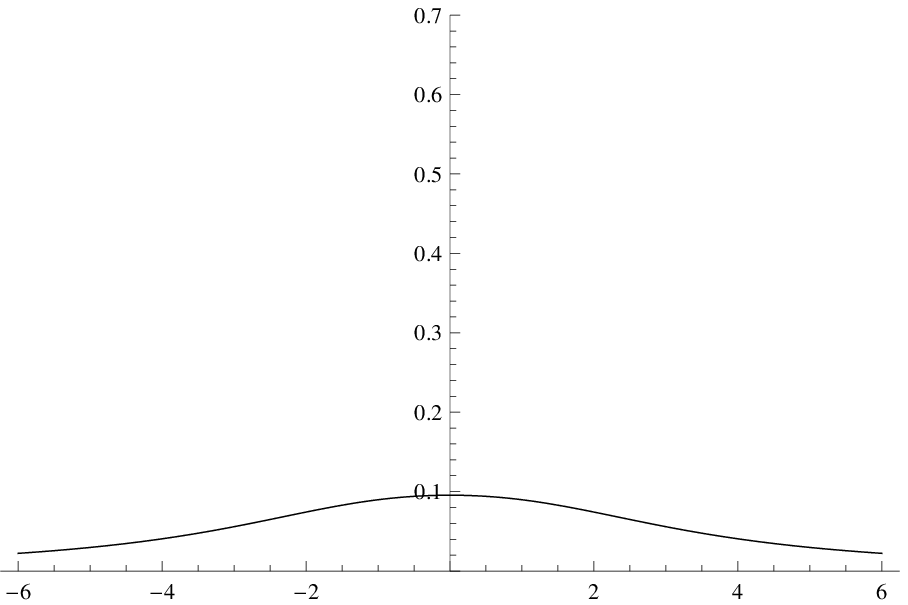}
\caption{$t=3$} \label{CP3}
\end{center}
  \end{minipage}
\end{center}
\end{figure}

 \begin{proposition}\label{c-Cauchy}
There exists a probability measure $\mu$ on $\mathbb{R}$ such that $\mu*C_t$ is not unimodal for any $t>0$, and 
\begin{equation}\label{moments}
\int_\mathbb{R} |x|^p \,d\mu(x)<\infty, \hspace{2mm} 0<p<3. 
\end{equation}
\end{proposition}
\begin{proof}
Let $\{w_n\}_{n\geq1}$ be a sequence of positive numbers such that $\sum_{n=1}^\infty w_n=1$ and $\{a_n\}_{n\geq1}$ be a sequence of real numbers. 
Consider the probability measure 
\begin{equation}
\mu= \sum_{n=1}^\infty w_n \delta_{a_n}. 
\end{equation}
Suppose that the sequence 
\begin{equation}
b_k=\inf_{n\in \N \setminus\{k\}}|a_k - a_n-1|, \qquad k \in \N
\end{equation}
 satisfies the condition
\begin{equation}\label{cond}
\lim_{k\to\infty} w_k b_k^3 =\infty. 
\end{equation}
When $\{a_n\}$ is increasing, this condition means that the distance between $a_{n}$ and $a_{n+1}$ grows sufficiently fast.   
Let 
\begin{equation}
f_t(x):= \frac{\pi}{t} \frac{d (\mu\ast C_t)}{d x} = \sum_{n=1}^\infty \frac{w_n}{(x-a_n)^2+t^2}. 
\end{equation}
Then we obtain 
\begin{equation}
f_t'(x) =  \sum_{n=1}^\infty \frac{-2w_n (x-a_n)}{[(x-a_n)^2+t^2]^2},  
\end{equation}
and so for each $k \in \N$ and each $t>0$
\begin{equation}
\begin{split}
f_t'(a_k-1) 
&= \frac{2 w_k}{(1+t^2)^2} + \sum_{n\geq1, n\neq k}\frac{-2w_n (a_k-a_n-1)}{[(a_k-a_n-1)^2+t^2]^2}\\
&\geq \frac{2 w_k}{(1+t^2)^2} - \sum_{n\geq1, n\neq k}\frac{2w_n}{|a_k-a_n-1|^3} \\
&\geq \frac{2 w_k}{(1+t^2)^2} - \sum_{n\geq1, n\neq k}\frac{2 w_n}{b_k^3} \\
&\geq \frac{2 w_k}{(1+t^2)^2} - \frac{2}{b_k^3} = 2w_k\left(\frac{1}{(1+t^2)^2}- \frac{1}{w_k b_k^3} \right).  
\end{split}
\end{equation}
The condition \eqref{cond} shows that $f_t'(a_k-1)$ is positive for sufficiently large $k \in \N$. This shows that $\mu \ast C_t$ is not unimodal for any $t>0$. 

If we take the particular sequences $a_n=a^n$ and $w_n= c n^r  a^{-3n}$, where $a\geq2, r>0$ and $c>0$ is a normalizing constant, then the sequence $b_n$ satisfies $b_n \geq C a^n$ for some constant $C>0$ independent of $n$. Then the conditions \eqref{cond} and \eqref{moments} hold true. 
\end{proof}

In the above construction, for any positive weights $\{w_n\}_n$ and any sequence $\{a_n\}_n$ that satisfies \eqref{cond}, the third moment of $\mu$ is always infinite,  
\begin{equation}
\int_{\mathbb{R}} |x|^3\,d\mu(x)=\infty, 
\end{equation}
 due to the inequality $|a_k| \geq b_k-|a_1|-1$ and the condition \eqref{cond}. 
The next question is then whether there exists a probability measure $\mu$ with a finite third moment such that $\mu \ast C_t$ is not unimodal for any $t>0$, or at least for sufficiently large $t>0$.  The complete answer is given below. 


\begin{theorem}\label{c-Cauchy2}
Let $\mu$ be a probability measure on $\R$ which has a finite absolute third moment 
$$
\beta := \int_\R |x|^3 \,d\mu(x) <\infty. 
$$ 
Then $\mu \ast C_t$ is unimodal for $t\geq20 \beta^{\frac{1}{3}}$. 
\end{theorem}
\begin{proof}
The  finite third moment condition implies the tail estimate 
\begin{equation}\label{tail}
\mu([-x,x]^c)= \int_{|y| >x} \,d\mu(y) \leq  \int_{|y|>x} \left(\frac{|y|}{x}\right)^3 \,d\mu(y) \leq \frac{\beta}{x^3}, \qquad x>0. 
\end{equation}
It suffices to prove that the function 
\begin{equation}
f_t(x):= \frac{\pi}{t} \frac{d (\mu\ast C_t)}{d x} = \int_\R \frac{1}{(x-y)^2+t^2}\,d\mu(y) 
\end{equation}
has a unique local maximum for large $t>0$.  Suppose that $x>0$ for some time. The derivative $f_t'$ splits into the positive and negative parts 
\begin{equation}
f_t'(x) =  \int_{y > x} \frac{2(y-x)}{[(x-y)^2+t^2]^2}\,d\mu(y)  - \int_{y \leq x} \frac{2(x-y)}{[(x-y)^2+t^2]^2}\,d\mu(y). 
\end{equation}
By calculus the function $u \mapsto \frac{2u}{(u^2+t^2)^2}$ takes a global maximum at the unique point $u=t/\sqrt{3}$. Using the tail estimate \eqref{tail} yields the estimate of the positive part 
\begin{equation}\label{positive}
\begin{split}
\int_{y > x} \frac{2(y-x)}{[(x-y)^2+t^2]^2}\,d\mu(y) 
&\leq \int_{y > x} \frac{2\frac{t}{\sqrt{3}}}{(\frac{t^2}{3}+t^2)^2}\,d\mu(y) \leq  \frac{\beta}{t^3 x^3}. 
\end{split}
\end{equation}
On the other hand the negative part can be estimated as 
\begin{equation}
\begin{split}
 \int_{y \leq x} \frac{2(x-y)}{[(x-y)^2+t^2]^2}\,d\mu(y)
 &\geq  \int_{-x < y < x/2 } \frac{2(x-y)}{[(x-y)^2+t^2]^2}\,d\mu(y). 
\end{split}
\end{equation}
Elementary calculus shows that for $-x<y<x/2$, 
\begin{equation}
\frac{2(x-y)}{[(x-y)^2+t^2]^2} \geq \min\left\{ \frac{4x}{(4x^2+t^2)^2}, \frac{x}{(x^2/4+t^2)^2} \right\}, 
\end{equation}
and if we further restrict to the case $x\geq t/4$, then 
\begin{equation}
\min\left\{ \frac{4x}{(4x^2+t^2)^2}, \frac{x}{(x^2/4+t^2)^2} \right\} \geq \min\left\{ \frac{4x}{(4x^2+16x^2)^2}, \frac{x}{(x^2/4+16x^2)^2} \right\} \geq \frac{10^{-3}}{x^3}. 
\end{equation}
Thus we obtain 
\begin{equation}\label{negative}
\begin{split}
 \int_{y \leq x} \frac{2(x-y)}{[(x-y)^2+t^2]^2}\,d\mu(y)
 &\geq \frac{10^{-3}}{x^3} \mu((-x/2,x/2)) \geq \frac{10^{-3}}{x^3} \left(1- \frac{\beta}{(x/2)^3}\right) \\
 &\geq  \frac{10^{-3}}{x^3}\left(1- \frac{8^3\beta}{t^3}\right), \qquad x \geq t/4. 
\end{split}
\end{equation}
Comparing \eqref{positive} and \eqref{negative}, taking $t \geq 20 \beta^{1/3}$ guarantees that the positive part is smaller than the negative part, and hence 
\begin{equation}
f_t'(x) <0, \qquad x \geq \frac{t}{4}. 
\end{equation}
Similarly, if $t \geq 20 \beta^{1/3}$ then 
\begin{equation}
f_t'(x) >0, \qquad x\leq -\frac{t}{4}. 
\end{equation}
In order to show that $f_t'$ has a unique zero, it suffices to show that $f_t''(x)<0$ for $|x| \leq t/4$. Now we have
\begin{equation}
f_t''(x)= \int_{|y -x| > t/\sqrt{3}} \frac{2[3(y-x)^2-t^2]}{[(x-y)^2+t^2]^3}\,d\mu(y)  - \int_{|y -x| \leq t/\sqrt{3}} \frac{2[t^2-3(y-x)^2]}{[(x-y)^2+t^2]^3}\,d\mu(y). 
\end{equation}
The function $u\mapsto 2(3u^2-t^2)/(u^2+t^2)^3$ attains a global maximum at $u=\pm t$ and a global minimum at $u=0$. Therefore, 
the positive part can be estimated as follows: 
\begin{equation}
\begin{split}
\int_{|y -x| > t/\sqrt{3}} \frac{2[3(y-x)^2-t^2]}{[(x-y)^2+t^2]^3}\,d\mu(y) 
&\leq \int_{|y -x| > t/\sqrt{3}} \frac{2(3t^2-t^2)}{(t^2+t^2)^3}\,d\mu(y) \\
& = \frac{1}{2t^4}\mu\left(\left[x-\frac{t}{\sqrt{3}}, x+\frac{t}{\sqrt{3}}\right]^c\right). 
\end{split}
\end{equation}
For all $x$ such that $|x| \leq t/4$, we have the inclusion 
\begin{equation}
\left[x-\frac{t}{\sqrt{3}}, x+\frac{t}{\sqrt{3}}\right]^c \subseteq \left[-\frac{t}{5}, \frac{t}{5}\right]^c, 
\end{equation}
and hence we obtain 
\begin{equation}\label{positive2}
\begin{split}
\int_{|y -x| > t/\sqrt{3}} \frac{2[3(y-x)^2-t^2]}{[(x-y)^2+t^2]^3}\,d\mu(y)  
&\leq  \frac{1}{2t^4} \mu\left( \left[-\frac{t}{5}, \frac{t}{5}\right]^c\right) \leq  \frac{5^3 \beta}{2t^7}. 
\end{split}
\end{equation}
On the other hand, the negative part has the estimate
\begin{equation}
\int_{|y -x| \leq t/\sqrt{3}} \frac{2[t^2-3(y-x)^2]}{[(x-y)^2+t^2]^3}\,d\mu(y) \geq \int_{|y -x| \leq t/2} \frac{2[t^2-3(y-x)^2]}{[(x-y)^2+t^2]^3}\,d\mu(y). 
\end{equation}
By calculus, the function $u\mapsto (t^2-3u^2)/(u^2+t^2)^3$ is decreasing on $[0,t]$, and so 
\begin{equation}\label{negative2}
\begin{split}
&\int_{|y -x| \leq t/2} \frac{2[t^2-3(y-x)^2]}{[(x-y)^2+t^2]^3}\,d\mu(y) \\
&\qquad \geq \int_{|y -x| \leq t/2} \frac{2(t^2-3\frac{t^2}{4})}{(\frac{t^2}{4}+t^2)^3}\,d\mu(y) = \frac{32}{125 t^4} \mu\left(\left[x-\frac{t}{2}, x+\frac{t}{2}\right]\right) \\
&\qquad \geq \frac{32}{125 t^4} \mu\left(\left[-\frac{t}{4}, \frac{t}{4}\right]\right) 
\geq  \frac{1}{4 t^4}\left(1- \frac{4^3\beta}{t^3}\right) 
\end{split}
\end{equation}
for all $|x| \leq t/4$. The positive part \eqref{positive2} is smaller than the negative part \eqref{negative2} if we take $t$ in such a way that $t \geq 10 \beta^{1/3}$. Thus $f_t''(x)<0$ for all $|x|\leq t/4$ and $t \geq 10 \beta^{1/3}$. 
\end{proof}


\subsection{Positive stable process with index $1/2$ with initial distribution}
A positive stable process with index $1/2$ has the distribution 
\begin{equation}
L_t(dx):=\frac{t}{\sqrt{2\pi}}\cdot \frac{e^{-\frac{t^2}{2x}}}{x^{3/2}} \cdot 1_{[0,\infty)}(x)\,dx, \hspace{2mm} x\in\mathbb{R}, 
\end{equation} 
at time $t\geq0$ which is called the {\it L\'evy distribution}.  We restrict to the case where the initial distribution is compactly supported. 

\begin{theorem}\label{c-Levy}
If $\mu$ is a compactly supported on $\mathbb{R}$ with diameter $D_\mu$ then $\mu\ast L_t$ is unimodal for all $t\ge (90/4)^{1/4} D_\mu^{1/2}$.
\end{theorem}

\begin{proof}
Considering the translation we may assume that $\mu$ is supported on $[0,\gamma]$, where $\gamma=D_\mu$. 
We set the following function:
\begin{equation}
g_{t,y}(x):=\frac{e^{-\frac{t^2}{2(x-y)}}}{(x-y)^{3/2}} 1_{[0,\infty)}(x-y), \qquad x, y\in \R.
\end{equation}
Note that this function is $C^\infty$ with respect to $x$. 
Consider the following function
\begin{equation}
f_t(x):=\frac{\sqrt{2\pi}}{t}\cdot\frac{d(\mu\ast L_t)}{dx}(x)=\int_{0}^{\gamma} g_{t,y}(x)\,d\mu(y), \qquad x \in \mathbb R,  
 \end{equation}
 which is supported on $(0,\infty)$ and has the derivative 
 \begin{equation}
 \frac{d}{d x} f_t(x)=\int_{0}^{\gamma} \frac{t^2-3(x-y)}{2(x-y)^{7/2}}e^{-\frac{t^2}{2(x-y)}}1_{[0,\infty)}(x-y)\,d\mu(y). 
\end{equation}
If $(0<)~ x < t^2/3$ then $t^2-3(x-y)< t^2-3 \cdot t^2/3+3=0$, and therefore $f_t'(x)>0$. Moreover, if $t^2/3+\gamma< x$ then we have that $t^2-3(x-y)< t^2-3(t^2/3+\gamma)+3\gamma=0$, and therefore $f_t'(x)<0$. For $t^2/3 < x < t^2/3 + \gamma$, the second derivative of $f_t$ is given by
\begin{equation}
f_t''(x)=\int_0^\gamma \frac{15(x-y)^2-10t^2(x-y)+t^4}{4(x-y)^{11/2}}e^{-\frac{t^2}{2(x-y)}}1_{[0,\infty)}(x-y)\,d\mu(y).
\end{equation}
Note $t^2/3 - \gamma < x - y < t^2/3 + \gamma$. 
Since for $X \in \mathbb R$
\begin{equation}
15X^2-10t^2X+t^4<0 \hspace{2mm} \text{ iff } \hspace{2mm} \frac{t^2}{3}-\frac{2 t^2}{3\sqrt{10}}<X<\frac{t^2}{3}+\frac{2 t^2}{3\sqrt{10}}, 
\end{equation}
if $\frac{2 t^2}{3\sqrt{10}}\geq \gamma$ then $f_t''(x)<0$. 

To summarize, we have obtained that if $t^2 \ge \frac{3\sqrt{10}}{2}\gamma$ then
\begin{itemize}
\item $f_t'(x)>0, \qquad x < t^2/3,$
\item $f_t'(x)<0, \qquad x>t^2/3+\gamma,$
\item $f_t''(x)<0, \qquad t^2/3<x<t^2/3+\gamma.$
\end{itemize}
Hence $f_t$ has a unique local maximum in $[t^2/3,t^2/3+\gamma]$. Therefore $\mu\ast L_t$ is unimodal for all $t^2\ge \frac{3\sqrt{10}}{2}\gamma$.
\end{proof}

We give a counterexample for large time unimodality for positive stable processes of index $1/2$ when initial distributions are not compactly supported.

\begin{proposition}\label{Levy}
There exists a probability measure $\mu$ on $\mathbb{R}$ such that $\mu\ast L_t$ is not unimodal for any $t>0$, and 
\begin{equation}
\int_\mathbb{R} |x|^p \,d\mu(x)<\infty, \qquad 0<p<\frac{5}{2}.
\end{equation}
\end{proposition}

\begin{proof}
Let $\{w_n\}_{n\ge1}$ be a sequence of positive numbers such that $\sum_{n=1}^\infty w_n=1$ and $\{a_n\}_{n\ge1}$ be a sequence of real numbers. Consider the probability measure
\begin{equation}
\mu=\sum_{n=1}^\infty w_n\delta_{a_n}.
\end{equation}
Suppose that the sequence
\begin{equation}
b_k=\inf_{n \in \mathbb N \setminus \{k\}} |a_k-a_n+1|, \qquad k\in\mathbb{N}
\end{equation}
satisfies the conditions
\begin{equation}\label{L-condi}
\lim_{k\to \infty}b_k=\infty, \qquad \lim_{k\to\infty}w_kb_k^{5/2}=\infty.
\end{equation}
Let
\begin{equation}
f_t(x):=\frac{\sqrt{2\pi}}{t}\frac{d(\mu\ast L_t)}{dx}=\sum_{n\ge1, a_n<x}w_n \frac{e^{-\frac{t^2}{2(x-a_n)}}}{(x-a_n)^{3/2}}.
\end{equation}
Then we obtain
\begin{equation}
f_t'(x)=\sum_{n\ge1, a_n<x}w_n \frac{[t^2-3(x-a_n)]}{2(x-a_n)^{7/2}} e^{-\frac{t^2}{2(x-a_n)},}
\end{equation}
and for each $k\in\mathbb{N}$ and each $t>0$
\begin{equation}
f_t'(a_k+1)=w_k\cdot\frac{t^2-3}{2}e^{-t^2/2} + \sum_{n: n\neq k, a_n < a_k+1} w_n \frac{[t^2-3(a_k-a_n+1)]}{2(a_k-a_n+1)^{7/2}} e^{-\frac{t^2}{2(a_k-a_n+1)}}.
\end{equation}
The condition \eqref{L-condi} implies that $b_k>\frac{5+\sqrt{10}}{15}t^2$ for all $k\ge k(t)$ where $k(t)$ is a some positive integer.  Since the map $x\mapsto \frac{t^2-3x}{x^{7/2}} e^{-\frac{t^2}{2x}}$ has a global minimum at $x=\frac{5+\sqrt{10}}{15}t^2$, for $k\ge k(t)$ we have
\begin{equation}
f_t'(a_k+1)\ge w_k\cdot\frac{t^2-3}{2}e^{-t^2/2}+ \frac{t^2-3b_k}{2b_k^{7/2}} e^{-\frac{t^2}{2b_k}}.
\end{equation}
The condition \eqref{L-condi} shows that $f_t'(a_k+1)$ is positive for sufficiently large $k\in\mathbb{N}$. This shows that $\mu \ast L_t$ is not unimodal for any $t>0$.

If we take the particular sequences $a_k=a^k$ and $w_k=cka^{-\frac{5}{2}k}$ where $a>2$ and $c>0$ is a normalizing constant, then the sequence $b_k$ satisfies $b_k \geq C a^k$ for some constant $C>0$. Then the condition \eqref{L-condi} holds true and 
\begin{equation}
\int_\mathbb{R} |x|^p\, d\mu(x)= \sum_{k\ge1} w_k |a_k|^p=c \sum_{k\ge1} ka^{(p-5/2)k}.
\end{equation}
Hence the above integration is finite if and only if $0<p<5/2.$
\end{proof}

In the above construction, for any positive weights $\{w_n\}_n$ and any sequence $\{a_n\}_n$ that satisfies \eqref{L-condi}, the $5/2$-th moment of $\mu$ is always infinite, that is,
\begin{equation}\label{eq:5/2}
\int_\mathbb{R} |x|^{5/2} \,d\mu(x)=\infty.
\end{equation}
We conjecture that if the integral in \eqref{eq:5/2} is finite then $\mu \ast L_t$ is unimodal in large time.  More generally, considering results on Cauchy processes in Section \ref{sec:Cauchy}, it is natural to expect the following.  
\begin{conjecture}
Suppose that $S_t$ is the law of a classical or free $\alpha$-stable process at time $t$, where $\alpha \in (0,2)$. If $\mu$ is a probability measure such that
\begin{equation}
\int_\mathbb{R} |x|^{2+\alpha}\, d\mu(x)<\infty, 
\end{equation}
then $\mu \ast S_t$ is unimodal for sufficiently large $t>0$. 
\end{conjecture}


\end{document}